\documentclass[reqno,10pt]{amsart}
\usepackage{mathtools,amssymb,wasysym}
\usepackage{mathrsfs}
\usepackage{enumerate}
\usepackage[usenames,dvipsnames]{color}
\usepackage{hyperref}
\hypersetup{colorlinks=true, pdfstartview=FitV, linkcolor=BrickRed,citecolor=BrickRed, urlcolor=BrickRed}
\definecolor{labelkey}{rgb}{0.6,0,0}

\usepackage[margin=1.2in]{geometry}
\numberwithin{equation}{section}
\usepackage[numbers,sort&compress]{natbib}
\usepackage{hyperref}
\usepackage{esint}
\usepackage{verbatim}

\theoremstyle{plain}
\newtheorem{theorem}{Theorem}[section]
\newtheorem{lemma}[theorem]{Lemma}

\newtheorem{proposition}[theorem]{Proposition}

\theoremstyle{definition}

\theoremstyle{remark}
\newtheorem{remark}[theorem]{Remark}
\newtheorem*{rem*}{Remark}

 \makeatletter
 \newcommand{\norm}{\@ifstar{\@normb}{\@normi}}
 \newcommand{\@normb}[2]{\left\Vert{#1}\right\Vert_{#2}}
 \newcommand{\@normi}[2]{\Vert{#1}\Vert_{#2}}
 \makeatother

\def\pr{{\partial}}

\def\les{\apprle}

\def\div{{\,\mathrm{div}\,}}

\def\R{{\mathbb R}}

\def\KK{{\mathcal K}}

\def\KK{{\mathcal K}}

\def\H{{\mathcal{H}}}

\def\R{{\bf R}}

\def\R{\mathbb{R}}

\def\T{{\mathcal{T}}}
 
 \global\long\def\Sob#1#2{{W}^{#1,#2}}

 \global\long\def\Leb#1{L^{#1}}

 \DeclareMathOperator{\arctanh}{arctanh}
 
 \newcommand{\action}[1]{\left<#1 \right>}
 
 \RequirePackage{bm}

 \newcommand{\boldj}{\mathbf{j}}

 \DeclareMathOperator{\Div}{div}
 \newcommand{\relphantom}[1]{\mathrel{\phantom{#1}}}
 
 \newcommand{\myd}[1]{\,\mathrm{d}{#1}}

 \DeclareMathOperator{\supp}{supp}

\newcommand{\of}[1]{\left(#1\right)} 
\newcommand{\off}[1]{\left[#1\right]}

\usepackage{tcolorbox}

\allowdisplaybreaks

\thanks{W. Huang was partially supported by NSF grant DMS-2452275. }

\begin{document}

\title[Scattering map for the Vlasov--Poisson system with a repulsive harmonic potential]{Scattering map for the Vlasov--Poisson system with a repulsive harmonic potential}

\author[W. Huang]{Wenrui Huang}
\address{Department of Mathematics, Brown University, 151 Thayer Street, Providence, RI 02912, USA}
\email{wenrui\_huang@brown.edu }

\author[H. Kwon]{Hyunwoo Kwon}
\address{Division of Applied Mathematics, Brown University, 182 George Street, Providence, RI 02912, USA}
\email{hyunwoo\_kwon@brown.edu }

\begin{abstract}
We consider the Vlasov--Poisson system with a repulsive harmonic potential and prove (modified) scattering of solutions for small initial data. Moreover, we construct (modified) wave operators for small final data. The main novelty of this work is to introduce the kinetic lens transform which enables us to relax the regularity assumption and the spatial decay assumption on the initial data compared to \cite{BVRVR25,VRVR24} and  to construct wave operators.
\end{abstract}

\maketitle

\section{Introduction}

The Vlasov--Poisson system is one of the fundamental equations in kinetic theory, widely studied in astrophysics and plasma physics to describe the collective dynamics of large-scale particle systems. It describes the distribution of stars in a galaxy \cite{J15} or electrons in a plasma \cite{V38} in a collisionless regime. It is also of interest to include the external potential in the system when the environment can affect the physical system beyond self-consistent interactions. Prominent examples are the effect of tidal forces on satellite galaxies \cite{HW99}, a model of the galactic halo \cite{DJ17}, and the formation of a star cluster \cite{CSFG14}, where phase mixing is induced by the background external potential. See \cite{BT08} for a detailed exposition. Hence, it is natural to ask which external potential leads to a satisfactory theory for well-posedness and the asymptotic dynamics of solutions to the system.

The Vlasov--Poisson system with an external potential $V$ is given by
\begin{equation}\label{eq:VP-external-potential}
\partial_t f +\{f,\mathcal{H}\}=0,
\end{equation}
where the Poisson bracket is defined as $\{f,g\}=\nabla_x f \cdot \nabla_v g-\nabla_v f \cdot \nabla_x g$ and the associated Hamiltonian $\mathcal{H}$ is given by
\begin{equation}\label{eq:force-field}
 \mathcal{H}=\frac{|v|^2}{2}+V(x)-\lambda \phi(t,x),\quad \phi(t,x)=\iint_{\mathbb{R}^d_y\times\mathbb{R}^d_v} \Gamma(x-y)f(t,y,v)\myd{v}\myd{y}.
\end{equation}
Here $\Gamma$ denotes the fundamental solution to $-\Delta$ and $f$ denotes a particle distribution function $f:\mathbb{R}\times\mathbb{R}^d_x\times\mathbb{R}^d_v\rightarrow\mathbb{R}_+$. The case $\lambda=-1$ corresponds to the gravitational case and $\lambda=1$ corresponds to the plasma case. When $V=0$, the system becomes the classical Vlasov--Poisson system. When $V(x)=-|x|^2/2$, the system becomes the Vlasov--Poisson system with a repulsive harmonic potential, which is the main interest of this paper.

We first review the Vlasov--Poisson system and the related system in the absence of the external potential. The global existence theory for the Vlasov--Poisson system is classical. When the initial data is sufficiently small, Bardos and Degond \cite{BD85} first proved the global existence of smooth solutions. Subsequently, Lions and Perthame \cite{LP91} and Pfaffelmoser \cite{P92} independently established the global existence of classical solutions for large initial data. Once we have a global solution to the problem, it is natural to seek the asymptotic behavior of the solutions. 

There has been substantial progress on the scattering problem for the Vlasov--Poisson system. Chae and Ha \cite{CH06} showed that the solution scatters linearly when $d\geq 4$, that is, the nonlinear solution converges to a solution of the linear kinetic transport equation. When $d=3$, due to the long-range effect of the Coulomb potential, Choi and Ha \cite{CH11} showed that any nontrivial classical solution to the Vlasov--Poisson system does not exhibit linear scattering. Later, Choi and Kwon \cite{CK16} proved that the small data solution will converge to the asymptotic state at infinity along a logarithmic implicit change of the linear trajectory.

Recently, many techniques from dispersive analysis have been adapted to kinetic equations. Ionescu, Pausader, Wang, and Widmayer \cite{IPWW22} characterized the implicit trajectory in \cite{CK16} via so-called $Z$-norm method which has been employed in several works to prove global stability of solutions to diverse dispersive equations. Later, Flynn, Ouyang, Pausader, and Widmayer \cite{FOPW23} obtained a similar result under weaker assumptions on the initial data, guided by the Hamiltonian structure and a pseudoconformal transformation. Moreover, they constructed wave operators and scattering maps. In the presence of the boundary, the first-named author, Pausader, and Suzuki \cite{HuPaSu} proved modified scattering of solutions in a smooth convex domain with perfectly conducting boundary conditions. We note that blow-up solutions to the Vlasov--Poisson system and its relativistic version \cite{LMR08,LMR08b} were shown by adapting an idea from dispersive analysis. 

From a different perspective, Pankavich \cite{Pan22} proved modified scattering for a multispecies case under the assumption that the electric field decays sufficiently fast, rather than imposing smallness on the initial data. More recently, Bigorgne and Velozo Ruiz \cite{BiVe} and Schlue and Taylor \cite{ScTa25} obtained refined asymptotic expansions describing the long-time behavior of solutions.  We also refer to \cite{PW21, PW24} for modified scattering results in the presence of a point charge.

There are also numerous results concerning scattering for other Vlasov-type systems. Bigorgne \cite{Bi25} and Pankavich and Ben-Artzi \cite{PaBe24} established modified scattering for the Vlasov--Maxwell system by different methods. A simplified proof was given by Breton \cite{B26}. When $\Gamma$ in \eqref{eq:force-field} is replaced with the Yukawa potentials, Iacobelli, Rossi, and Widmayer \cite{IaRoKl26} proved scattering for the screened Vlasov--Poisson system in dimensions $d\ge 2$, and Wei \cite{Wei25} completed the analysis in the one-dimensional case. When $\Gamma$ is replaced with the Riesz kernel of order $\alpha$, the authors \cite{HK25} studied the Vlasov--Riesz systems of order $\alpha$ and proved modified scattering with a polynomial correction for $1<\alpha<1+\delta$, while linear scattering holds in the regime $1/2<\alpha<1$ which was previously observed in \cite{CH11}. 

To describe phase mixing for a collisionless gas \cite{RS20} and the formation of a star cluster \cite{CSFG14}, it is reasonable to introduce external potentials $V$ to the Vlasov--Poisson system. If we have an external smooth potential $V$, the second-order linearization around the stationary point will naturally induce harmonic potential terms $\pm |x|^2/2$. One example is the logarithmic halo potential $V(x)=\ln(1+|x|^2)$, which shares many features of elliptical galaxies \cite{BBP07}.

For the confining case, Chaturvedi and Luk \cite{CL22,CL24} have investigated phase mixing for the 1D linear kinetic transport equation with external potential $V(x)=x^2/2+\varepsilon x^4/2$, $\varepsilon>0$ and gravitational Vlasov--Poisson with an external Kepler potential $V(x)=-|x|^{-1}$, respectively. When we have repulsive harmonic potentials, $V(x)=-|x|^2/2$, Velozo Ruiz and Velozo Ruiz \cite{VRVR24} first established sharp space--time decay estimates for small-data solutions and proved the global existence of small solutions to \eqref{eq:VP-external-potential} in dimensions $d\geq 2$ using commuting vector fields and modified vector field methods. Later, Bigorgne, Velozo Ruiz, and Velozo Ruiz \cite{BVRVR25} proved modified scattering of solutions in the two-dimensional case using a refined modified vector field method in a natural coordinate that captures the hyperbolicity of the linearized system.

The purpose of this paper is to show small-data modified scattering of solutions and construct a modified wave operator for the problem \eqref{eq:VP-rep-potential-reformulate}. Compared to \cite{VRVR24,BVRVR25}, our main result does not require a higher regularity assumption on the initial data, and we do not assume spatial polynomial decay on the initial data. 

\subsection{Main result} To state our main result, if we write $f=\mu^2$, then the problem \eqref{eq:VP-external-potential} with $V(x)=-|x|^2/2$ can be reformulated into 
\begin{equation}\label{eq:VP-rep-potential-reformulate}
\partial_t \mu + v\cdot \nabla_x \mu+x\cdot \nabla_v\mu +\lambda E[\mu]\cdot\nabla_v\mu=0,
\end{equation}
where 
\[
	E[\mu](t,x)=\iint_{\mathbb{R}^d_y\times\mathbb{R}^d_v} \nabla_x\Gamma(x-y)\mu^2(t,y,v)\myd{v}\myd{y}. 
\]

We will mainly discuss the case $d=2$ here since our method can easily deduce the results for the higher dimensions and the dynamics are simpler. See Remark \ref{rem:linear-scattering}.

\begin{theorem}\label{thm:A}
There exists $\varepsilon_0>0$ such that the following result holds: given $\mu_0 \in C^1_{x,v}$ satisfying 
\begin{equation}\label{eq:initial-mu}
\norm{\action{v}\mu_0}{\Leb{2}_{x,v}}+\norm{\action{v}\nabla_{x,v}\mu_0}{\Leb{\infty}_{x,v}}+\norm{\action{v}^3\mu_0}{\Leb{\infty}_{x,v}}\leq \varepsilon_0,
\end{equation}
there exists a unique global solution $\mu$ of \eqref{eq:VP-rep-potential-reformulate} with $\mu(t=0)=\mu_0$. Moreover, there exist $\mu_\infty\in \Leb{2}_{x,v}\cap\Leb{\infty}_{x,v}$ and $E_\infty=E[\mu_\infty]\in L^\infty_x$ such that we have 
\begin{equation}\label{eq:scattering}
\norm{\mu(t,\mathcal{X}(t,x-v,v),\mathcal{V}(t,x-v,v))-\mu_\infty(x,v)}{L^\infty_{x,v}}\apprle \frac{\action{t}^{10}}{e^t} \quad \text{as } t\rightarrow\infty,
\end{equation}
where 
\begin{equation}\label{eq:scattering-trajectory}
	\begin{aligned}
 \mathcal{X}(t,x,v)&=x \cosh t +v\sinh t -\lambda t e^{-t} E_\infty(x+v), \\
 \mathcal{V}(t,x,v)&=v \cosh t +x\sinh t +\lambda t e^{-t} E_\infty(x+v).
\end{aligned}
\end{equation}

\end{theorem}

\begin{remark}\leavevmode
\begin{enumerate}[(i)\,\,\,]
\item The global well-posedness of  \eqref{eq:VP-rep-potential-reformulate} and long-time dynamics were studied in \cite{BVRVR25,VRVR24} by means of (modified and commuting) vector field methods. In contrast, our approach allows for weaker assumptions on the initial data $\mu_0$. In particular, we remove the weight conditions in $x$ and require only one derivative on the initial data. We emphasize, however, that our hypotheses are formulated for $\mu_0$ with $f_0=\mu_0^2$ rather than directly for the distribution function $f_0$ and $f_0\in W^{2,\infty}_{x,v}$ implies only $\mu_0\in W^{1,\infty}_{x,v}$. 
\item Our trajectory \eqref{eq:scattering-trajectory} recovers the result in \cite{BVRVR25} once we change $x$ by $x+v$ and redefine $\mu_\infty$. We choose such a trajectory because it is easier to construct the wave operator and the relation between the asymptotic electric field $E_{\infty}$ and $\mu_{\infty}$ is simpler: $E_{\infty}=E[\mu_{\infty}]$.
\item The decay of the density and of the electric field, together with their derivatives, is transparent in view of \eqref{eq:computation-rules-lens}.  
\item Unlike the Vlasov--Poisson system without potentials (see e.g. \cite{FOPW23,IPWW22}), we do not have to switch the role of $x$ and $v$ in the asymptotic quantity $E_{\infty}$. 
\end{enumerate}
\end{remark}

Our next theorem concerns the existence of a modified wave operator for small data.
\begin{theorem}\label{thm:B}
There exists $\varepsilon_0>0$ such that given  $\mu_{\infty}\in W_{x,v}^{2,\infty}$ and $E_{\infty}=E[\mu_{\infty}]\in W^{3,\infty}$ satisfying
\begin{equation}\label{eq:condition-mu-infinity}
	\begin{split}
&\norm{\action{v}\mu_\infty}{\Leb{2}_{x,v}}+\norm{\action{v}^4\mu_\infty}{\Leb{\infty}_{x,v}}+\norm{\action{v}\nabla_{x,v} \mu_\infty}{\Leb{\infty}_{x,v}}\\
&\relphantom{=}+\norm{\action{v}^2\nabla^2_{x,v}\mu_\infty}{\Leb{\infty}_{x,v}}+\norm{E_\infty}{W^{3,\infty}}\leq \varepsilon_0,		
	\end{split}
\end{equation}
there exists a unique global strong solution of \eqref{eq:VP-rep-potential-reformulate} such that
\begin{equation}\label{eq:asymptotic-local-uniform}
		\mu(t,\mathcal{X}(t,x-v,v),\mathcal{V}(t,x-v,v))\rightarrow\mu_{\infty}(x,v)\quad \text{as } t\rightarrow \infty.
\end{equation}
\end{theorem}

\begin{remark}\leavevmode
\begin{enumerate}[(i)\,\,\,]
\item Unlike Theorem \ref{thm:A}, we need to impose stronger assumptions on the final data.
\item We only have local uniform (or pointwise) convergence here, since we can only propagate the bounds of  $\norm{\theta \nabla_z\sigma(s=0)}{L^\infty_{z,w}}=\norm{\nabla_v\mu (t=0)}{L^\infty}$ (see Theorem \ref{thm:wave-operator}).
\end{enumerate}
\end{remark}

As a consequence, we can define the scattering operator in the neighborhood of the origin $\mathcal{S}:\mu_{-\infty}\rightarrow\mu_{+\infty}$ as in the following theorem:
\begin{theorem}\label{thm:c}
There exists $\varepsilon_0>0$ such that given any asymptotic state $\mu_{-\infty}\in \Sob{2}{\infty}_{x,v}$ with $E_{-\infty}=E[\mu_{-\infty}]\in W^{3,\infty}$ satisfying 
	the condition \eqref{eq:condition-mu-infinity},
there exists a unique strong solution $\mu$ of \eqref{eq:VP-rep-potential-reformulate}, $\mu_{+\infty}\in L^2_{x,v}\cap L^\infty_{x,v}$, and $E_{+\infty}=E[\mu_{+\infty}]\in L^\infty_x$ such that
\[   \mu(t,\mathcal{X}_{\pm}(t,x\mp v,v),\mathcal{V}_{\pm}(t,x\mp v,v))\rightarrow \mu_{\pm\infty}(x,v)\quad \text{as } t\rightarrow\pm\infty,       \]
where $\mathcal{X}_{+},\mathcal{V}_{+}$ corresponds to $E_{+\infty}$ in \eqref{eq:scattering-trajectory} and 
\begin{equation*}\label{eq:scattering-trajectory2}
	\begin{aligned}
 \mathcal{X}_-(t,x,v)&=x \cosh t +v\sinh t +\lambda t e^{t} E_{-\infty}(x-v), \\
 \mathcal{V}_-(t,x,v)&=v \cosh t +x\sinh t +\lambda t e^{t} E_{-\infty}(x-v).
\end{aligned}
\end{equation*} 
\end{theorem}

\subsection{Idea of the proofs}
Our approach is motivated by previous works on the Vlasov--Poisson system, which crucially used the pseudoconformal symmetries \cite{LMR08,FOPW23,HK25} and the Hamiltonian structure of the system. However, this method cannot be applied directly to our problem, as the harmonic potential introduces new difficulties in the pseudoconformal transform, and the transformed equation is more delicate to analyze and requires spatial decay of the initial data.

Motivated by the lens transformation in the nonlinear Schr\"odinger equation \cite{Tao09}, we introduce a kinetically adapted hyperbolic lens transform \eqref{eq:kinetic-lens-transform} to convert the original problem \eqref{eq:VP-rep-potential-reformulate} into the following system
\begin{equation}\label{eq:VP-lens-intro}
\partial_s \gamma+p\cdot\nabla_q\gamma+(1-s^2)^{(d-4)/2}\lambda E[\gamma]\cdot\nabla_p\gamma=0,\quad s\in (-1,1).
\end{equation} 
This lens transform enables us to absorb the $x\cdot \nabla_v\mu$ term, which does not appear in the classical Vlasov-Poisson system. 

To prove Theorem \ref{thm:A}, we first obtain a local well-posedness result (Theorem \ref{thm:lwp}) for the transformed problem \eqref{eq:VP-lens-intro}. Unlike \cite{VRVR24,BVRVR25}, we do not need to assume the spatial decay assumption on the initial data. Once we convert the problem, we construct a local-in-time solution by the classical Cauchy-Lipschitz theorem. To show small data global well-posedness and scattering of the problem \eqref{eq:VP-lens-intro}, we rewrite the equation as a nonlinear transport equation on $[0,1)$. Then, by a bootstrap argument (Proposition \ref{prop:bootstrap}) with continuity estimates on the electric fields (Proposition \ref{prop:LipboundE}), we show that the velocity moments of the solution propagate, and this enables us to show modified scattering of solutions. 

To construct a wave operator and a scattering map (Theorems \ref{thm:B} and \ref{thm:c}) from the small data, we also use the hyperbolic lens transformation. Similar to the pseudoconformal transformation, the lens transformation also preserves the Hamiltonian structure. The transformation converts the original problem into the local well-posedness problem starting from $s=1$ or $s=-1$. Then one can try to construct a solution by employing a Picard iteration scheme similar to the local well-posedness theorem (Theorem \ref{thm:lwp}). However, as one can see in \eqref{eq:VP-lens-intro}, the problem has a singularity at those points $s=\pm 1$ and unlike the local well-posedness theorem (Theorem \ref{thm:lwp}), this singularity prevents us from constructing approximate solutions to the problem.

To mitigate this singularity, we introduce a change of coordinates \eqref{eq:change-of-variable-sigma} that preserves the symplectic structure and induces a new Hamiltonian $\mathcal{K}$ that is less singular at those points. Then the problem is reduced to showing the local well-posedness for this new Hamiltonian PDE
\begin{equation}\label{eq:new-sigma-PDE}
\partial_s \sigma + \{\sigma,\mathcal{K}\}=0,
\end{equation}
which still has a singularity at the ``initial'' time. In this new formulation, $\nabla_w\mathcal{K}$ and $\nabla_z\mathcal{K}$ are locally Lipschitz in phase variables and integrable in time near the initial time (Lemma \ref{lemma:bound-K}). This enables us to apply the Carath\'eodory extension of the Cauchy-Lipschitz theorem to show the existence of approximate solutions \eqref{eq:sigma-n-approx}. 

To obtain a priori estimates for derivatives of the approximate solutions, we use a weight function $\theta$ which compensates for the blow-up of the time variable at ``initial'' time, enabling us to get a priori estimates for solutions (Lemma \ref{lem:particle-density-bootstrap}). Then the solution can be constructed via a standard Picard iteration scheme, which implies Theorem \ref{thm:B}. Theorem \ref{thm:c} will follow by combining Theorems \ref{thm:A} and \ref{thm:B}.

\subsection{Organization} 
The rest of the paper is organized as follows. In Section \ref{sec:prelim}, we introduce the lens transform for the Vlasov--Poisson system. We also collect several auxiliary lemmas that will be used throughout the paper. Section \ref{sec:LWP} is devoted to the proof of the local well-posedness result for the system \eqref{eq:VP-lens-intro}. In Section \ref{sec:scattering}, we establish Theorem \ref{thm:A} via a bootstrap argument. Finally, in Section \ref{sec:wave}, we construct the wave operator and prove Theorems \ref{thm:B} and \ref{thm:c}.

We finish the introduction with notations. Let $\mathbb{R}^d$ denote the standard Euclidean space of points $x=(x^1,\dots,x^d)$ and let $\action{x}=(1+|x|^2)^{1/2}$. Let $B_R$ be the Euclidean open ball with radius $R$ centered at the origin. By $C_c^\infty$ the space of all smooth functions with compact support.  For two nonnegative quantities $A$ and $B$, we write $A\apprle_{\alpha,\beta,\dots} B$ if $A\leq CB$ for some positive constant $C$ that depends on the parameters $\alpha$, $\beta$, .... If the dependence is evident, then we usually omit the subscripts and denote $A\apprle B$. We write $A\approx B$ if $A\apprle B$ and $B\apprle A$. Finally, we frequently suppress the time variable when it involves norms, e.g., 
\[ \norm{\mu(s)}{\Leb{r}_{q,p}}=\norm{\mu}{\Leb{r}_{q,p}}.\]

\section{Preliminaries}\label{sec:prelim}

In this section, we first introduce the lens transform for the Vlasov--Poisson system, which connects it to the Vlasov--Poisson system with harmonic potentials. Next, we provide estimates on the electric fields, which are crucial for establishing global well-posedness of the problem and the scattering of solutions. Finally, we recall interpolation inequalities that involve weights and derivatives.

\subsection{Hyperbolic lens transform}

Inspired by the lens transformation in the nonlinear Schr\"odinger equation with harmonic potentials, we define the hyperbolic lens transform adapted to the Vlasov--Poisson system with repulsive harmonic potentials (see e.g. \cite{Tao09}).

For $(t,x,v)\in \mathbb{R}\times\mathbb{R}^d\times\mathbb{R}^d$, we define 
\begin{equation}\label{eq:kinetic-lens-transform}
 \mathcal{T}:(t,x,v)\mapsto \left(\tanh t,\frac{x}{\cosh t},v\cosh t-x\sinh t\right) 
\end{equation}
and its inverse
\[ \mathcal{T}^{-1}:(s,q,p)\mapsto \left(\arctanh s,\frac{q}{\sqrt{1-s^2}},\frac{sq}{\sqrt{1-s^2}}+p\sqrt{1-s^2}\right),\]
where $(s,q,p)\in (-1,1)\times\mathbb{R}^d\times \mathbb{R}^d$. The Jacobian of $\mathcal{T}$ is given by 
\begin{equation}\label{eq:Jacobian} \frac{\partial (q,p)}{\partial (x,v)}=\begin{bmatrix*}
\dfrac{1}{\cosh t}I_d & 0 \\[2ex]
-(\sinh t)I_d& (\cosh t)I_d
\end{bmatrix*},
\end{equation}
where $I_d$ denotes the $d\times d$ identity matrix.

If we write $\mu=\gamma\circ\mathcal{T}$, then a direct computation gives
\begin{equation}\label{eq:computation-rules-lens}
\begin{aligned}
\partial_t \mu &=\frac{1}{\cosh^2 t}\partial_s\gamma-\frac{\sinh t}{\cosh^2 t}x\cdot\nabla_q \gamma+(v\sinh t-x\cosh t)\cdot \nabla_p \gamma,\\
\nabla_x\mu&=\frac{1}{\cosh t} \nabla_q\gamma-(\sinh t)\nabla_p\gamma,\\
\nabla_v\mu&=(\cosh t)\nabla_p\gamma,\\
E[\mu](t,x)&=\frac{1}{(\cosh t)^{d-1}} E[\gamma]\left(\tanh t,\frac{x}{\cosh t}\right).
\end{aligned}
\end{equation}

 By \eqref{eq:Jacobian} and \eqref{eq:computation-rules-lens}, one can also see that the lens transform is symplectic  and  volume-preserving. Moreover, $\mu$ solves \eqref{eq:VP-rep-potential-reformulate} if and only if $\gamma$ solves
\begin{equation}\label{eq:VP-lens}
\partial_s \gamma+p\cdot\nabla_q\gamma+(1-s^2)^{(d-4)/2}\lambda E[\gamma]\cdot\nabla_p\gamma=0,\quad s\in (-1,1),
\end{equation}
which can be also rewritten in the Hamiltonian form
\begin{equation}\label{eq:Hamiltonian-gamma} 
\pr_s\gamma+\{\gamma,\H\}=0, \quad \mathcal{H}=\frac{|p|^2}{2}-\lambda (1-s^2)^{(d-4)/2}\phi(s,q).		 
\end{equation}
Note that $(1-s^2)^{(d-4)/2}$ is integrable on $(-1,1)$ if $d\geq 3$. If we can show that the electric field is bounded, then we will easily get the scattering result, see Remark \ref{rem:linear-scattering}. For this reason, we restrict our attention to $d=2$. 

From now on, for simplicity, we define the function $f:(-1,1)\rightarrow\mathbb{R}$ by 
\[ f(s)=\frac{1}{1-s^2} \]
and its antiderivative 
\[ F(s)=\arctanh s = \frac{1}{2}\ln \left(\frac{1+s}{1-s}\right),\quad -1<s<1.\]
\begin{remark}
	One can also replace all the hyperbolic functions with trigonometric functions and get  (trigonometric) lens transform: for $(t,x,v)\in (-\pi/2,\pi/2)\times\R^d\times\R^d$, we define
$$ \widetilde{\T}:(t,x,v)\mapsto \of{ \tan t,\frac{x}{\cos t}, x\sin t+v\cos t}   .$$
This transform is also symplectic and if we write $\mu=\gamma\circ \widetilde{\T}$, then a direct computation shows that $\mu$ solves the Vlasov--Poisson system with an attractive harmonic potential
\begin{equation*} 
		\partial_t \mu + v\cdot \nabla_x \mu-x\cdot \nabla_v\mu +\lambda E[\mu]\cdot\nabla_v\mu=0
\end{equation*}
if and only if $\gamma$ solves
\begin{equation}\label{eq:VP-lens2}
	\pr_s\gamma+p\cdot\nabla_q\gamma+ (1+s^2)^{(d-4)/2}\lambda E[\gamma]\cdot\nabla_p\gamma=0	.
\end{equation}  
\end{remark}
\begin{remark}
	By applying the inverse lens transform, one can reformulate the Vlasov--Poisson system into a time-dependent Vlasov--Poisson system with an attractive harmonic potential. A key advantage of this transformation is that it compactifies the time interval from $(-\infty, \infty)$ to $(-\pi/2, \pi/2)$.
\end{remark}

We will frequently use the following lemma throughout this paper.
\begin{lemma}\label{lem:Liouville}
Suppose that $r\in [1,\infty]$ and $\gamma \in C([0,T];\Leb{r}_{q,p})$ is a solution to 
\[
\partial_s \gamma+ \{\gamma,\mathcal{K}\}=g,\quad \gamma(s=0)=\gamma_0
\]
for some $\gamma_0 \in \Leb{r}_{q,p}$. Then we have 
\[ \norm{\gamma(s)}{\Leb{r}_{q,p}}\leq \norm{\gamma_0}{\Leb{r}_{q,p}}+\int_0^s \norm{g(\tau)}{\Leb{r}_{q,p}}\myd{\tau},\quad s\in [0,T].\]
\end{lemma}
\begin{proof}
Define $V^i=\partial_{p^i}\mathcal{K}$ and $V^{i+d}=-\partial_{q^i}\mathcal{K}$, $i=1,\dots, d$. Then $\Div_{q,p} V=0$ and $\gamma$ satisfies 
\[ \partial_s \gamma+ \Div_{q,p}(\gamma V)=g,\quad \gamma(s=0)=\gamma_0.\]
Then the result follows from the method of characteristics and Liouville's theorem.
\end{proof}

\subsection{Estimates on the electric fields}
To estimate the electric field, we will use the following decomposition of the electric field throughout this paper: we have
\begin{equation}\label{eq:electric-field-rep}
	E(s,q)=\int_0^\infty E_R(s,q)\frac{\myd{R}}{R^2},	
\end{equation}
where
\begin{equation*}
		E_R(s,q)=c_1\iint_{\mathbb{R}^2_y\times\mathbb{R}^2_v} (\nabla\chi)(R^{-1}(q-y))\gamma^2(s,y,v)\myd{y}\myd{v}, 
\end{equation*}
for some constant $c_1>0$ and $\chi \in C_c^\infty(B_2)$ is a radially symmetric and  nonincreasing function  such that 
$$ 0\leq \chi\leq 1, \quad   \chi\equiv 1 \,\text{ on } \,B_{1/2}.$$ Indeed, this decomposition follows from
\begin{equation*} 
-\frac{x}{|x|^2}=c\int_0^\infty R^{-1}(\nabla\chi)(R^{-1}x)\frac{\myd{R}}{R}
\end{equation*}
for some constant $c>0$.  Sometimes, we will also need to decompose the electric field in the velocity variable. Define
\begin{equation*}
	\begin{split}
		\widetilde{\chi}(v):=-v\cdot \nabla\chi(v).
	\end{split}
\end{equation*}
Then $\widetilde{\chi}$ is nonnegative, smooth and supported in $B_{2}\setminus B_{1/2}$,  which leads to
\begin{align*}
	E_{R,V}(s,q):=\iint_{\mathbb{R}^2_y\times\mathbb{R}^2_v} (\nabla\chi)(R^{-1}(q-y))\widetilde{\chi}(V^{-1}v)\gamma^2(s,y,v)\myd{y}\myd{v}
\end{align*}
with 
\begin{equation}\label{eq:electric-field-rep-2}
	E(s,q)=\int_0^\infty\int_0^\infty E_{R,V}(s,q)\frac{\myd{R}}{R^2}\frac{\myd{V}}{V}.
\end{equation}

Since $\supp \nabla\chi \subset B_2\setminus B_{1/2}$, we can easily show the following bounds (see e.g. \cite[Section 2.3]{HK25}):
\begin{lemma}\label{lem:estimate-E_RV}
Fix $s\in (-1,1)$ and $q\in\R^2$. Then we have
\begin{align*}
&|E_R(s,q)|\apprle \norm{\gamma}{L^2_{q,p}}^2,\quad |\nabla E_R(s,q)|\apprle R^{-1}\norm{\gamma}{L^2_{q,p}}^2, \\	
&|E_{R,V}(s,q)|\apprle R^2\min\{V^2\norm{\gamma}{L^\infty_{q,p}}^2,V^{-2}\norm{|p|^2\gamma}{L^\infty_{q,p}}^2\},\\
&|\nabla E_{R,V}(s,q)|\apprle R\min\{RV^2\norm{\gamma}{L^\infty_{q,p}}\norm{\nabla_q\gamma}{L^\infty_{q,p}},V^{-2}\norm{|p|^2\gamma}{L^\infty_{q,p}}^2\}
\end{align*}
for all $R,V>0$.
\end{lemma}
As an immediate consequence, we can bound the electric field by weighted norm of $\gamma$.

\begin{lemma}\label{lem:estimate-E}
For any $A>0$, $s\in(-1,1)$, and $q\in \mathbb{R}^2$, we have 
\begin{equation}\label{eq:boundE}
|E(s,q)|\apprle A^{-1} \left[\norm{\gamma}{\Leb{2}_{q,p}}^2+\norm{\gamma}{\Leb{\infty}_{q,p}}^2\right]+A^{3} \norm{|p|^2 \gamma}{\Leb{\infty}_{q,p}}^2.
\end{equation}
Also, for $0<\theta<1/2$, we have 
\begin{equation}\label{eq:bound-dE}
\begin{aligned}
|\nabla_q E(s,q)|&\apprle A\norm{\gamma}{\Leb{2}_{q,p}}^2+A^{-1/2+\theta}\norm{\gamma}{\Leb{\infty}_{q,p}}\norm{\nabla_q \gamma}{\Leb{\infty}_{q,p}}+A^{-\theta}\norm{|p|^2\gamma}{\Leb{\infty}_{q,p}}^2.
\end{aligned}
\end{equation}
\end{lemma}

\begin{remark}
In particular, if $A=\action{F(s)}^4$ and $\theta=1/5$, then we have
\begin{equation}\label{eq:gradient-E-control}
\begin{aligned}
|\nabla E(s,q)|&\apprle \action{F(s)}^4\norm{\gamma}{\Leb{2}_{q,p}}^2+\action{F(s)}^{-6/5} \norm{\gamma}{\Leb{\infty}_{q,p}}\norm{\nabla_q \gamma}{\Leb{\infty}_{q,p}}\\
&\relphantom{=}+\action{F(s)}^{-4/5}\norm{|p|^2\gamma}{\Leb{\infty}_{q,p}}^2.
\end{aligned}
\end{equation}
\end{remark}

\begin{proof}
By Lemma \ref{lem:estimate-E_RV}, we deduce that
\begin{align*}
|E(s,q)|&\apprle\int_{R=A}^\infty|E_R|\frac{\myd{R}}{R^2}+\int_{R=0} ^A\int_{V=0}^B |E_{R,V}|\frac{\myd{V}}{V}\frac{\myd{R}}{R^2}+\int_{R=0}^A\int_{V=B}^\infty|E_{R,V}|\frac{\myd{V}}{V}\frac{\myd{R}}{R^2}\\
&\apprle A^{-1}\norm{\gamma}{L^2_{q,p}}^2+AB^2\norm{\gamma}{L^\infty_{q,p}}^2+AB^{-2}\norm{|p|^2\gamma}{L^\infty_{q,p}}^2.	
\end{align*}
By choosing $B=A^{-1}$, we get \eqref{eq:boundE}.

To show \eqref{eq:bound-dE}, it follows from Lemma \ref{lem:estimate-E_RV} and $0<\theta<1/2$ that
\begin{align*}
|\nabla E(s,q)|&\apprle \int_{R=A^{-1/2}}^\infty |\nabla E_R| \frac{\myd{R}}{R^2}+\int_{R=0}^{A^{-1/2}}\int_{V=0}^{R^{-\theta}} |\nabla E_{R,V}|\frac{\myd{V}}{V}\frac{\myd{R}}{R^2}\\
&\relphantom{=}+\int_{R=0}^{A^{-1/2}}\int_{V=R^{-\theta}}^\infty |\nabla E_{R,V}|\frac{\myd{V}}{V}\frac{\myd{R}}{R^2}\\
&\apprle A\norm{\gamma}{\Leb{2}_{q,p}}^2+A^{-1/2+\theta}\norm{\gamma}{\Leb{\infty}_{q,p}}\norm{\nabla_q\gamma}{\Leb{\infty}_{q,p}}+A^{-\theta} \norm{|p|^2\gamma}{\Leb{\infty}_{q,p}}^2.
\end{align*}
This completes the proof of Lemma \ref{lem:estimate-E}.
\end{proof}

\subsection{Interpolation inequalities}

We will use the following interpolation inequalities in the construction of wave operators.
\begin{proposition}\label{prop:interpolation}
Let $d\geq 2$ and $\ell\geq 0$. Suppose that $f:\mathbb{R}^d_x\times\mathbb{R}^d_v\rightarrow\mathbb{R}$. 
\begin{enumerate}[\rm (i)\,]
\item We have
\[ \norm{\action{x}^\ell\nabla_v f}{\Leb{\infty}_{x,v}}\apprle \norm{\action{x}^{2\ell}f}{\Leb{\infty}_{x,v}}+\norm{\nabla_v^2f}{\Leb{\infty}_{x,v}}.\]
In particular, when $\ell=0$, we have
\[ \norm{Df}{\Leb{\infty}_{x,v}}\apprle \norm{f}{\Leb{\infty}_{x,v}}^{1/2}\norm{D^2f}{\Leb{\infty}_{x,v}}^{1/2},\quad D\in\{\nabla_x,\nabla_v\}.\]
\item We have
\[ \norm{\action{x}^\ell \nabla_x f}{\Leb{\infty}_{x,v}}\apprle \norm{\action{x}^{2\ell}f}{\Leb{\infty}_{x,v}}+\norm{\nabla^2_x f}{\Leb{\infty}_{x,v}}.\]
\end{enumerate} 
\end{proposition}
\begin{proof}
(i) can be proved by using Littlewood--Paley projection as in \cite[Proposition B.1]{HK25}. Although (ii) was proved in \cite[Proposition B.1]{HK25} in a restricted regime $d/2<\ell<d+1$, we give an elementary proof for general $\ell\geq 0$. The inequality only concerns $x$-weight and $x$-derivative; we may focus on the case $f:\mathbb{R}^d\rightarrow\mathbb{R}$. By Taylor's theorem, we have 
\[
f(x+he_j)=f(x)+h(\partial_{x_j} f)(x)+\int_0^1 (1-\theta)(\partial_{x_j x_j} f)(x+h\theta e_j)h^2 \myd{\theta}
\]
for $h>0$. If we choose $h=\action{x}^{-\ell}$, then 
\begin{align*}
\left|\action{x}^\ell \partial_{x_j} f(x)\right|&\leq \left|\action{x}^{2\ell} f(x+\action{x}^{-\ell}e_j)\right|+\left|\action{x}^{2\ell} f(x)\right|+\norm{\nabla_x^2 f}{\Leb{\infty}_x}\\
&\leq \left(\sup_x \frac{\action{x}^{2\ell}}{\action{x+\action{x}^{-\ell}e_j}^{2\ell}}+1\right)\norm{\action{x}^{2\ell} f}{\Leb{\infty}}+\norm{\nabla_x^2 f}{\Leb{\infty}_x}
\end{align*}
for all $x\in\mathbb{R}^d$ and $j=1,\dots,d$. 
Clearly, ${\action{x}^{2\ell}}/{\action{x+\action{x}^{-\ell}e_j}^{2\ell}}$ is a bounded function on $\mathbb{R}^d$. This implies the desired result.
\end{proof}

\section{Local well-posedness of the Vlasov--Poisson system with the repulsive potential}\label{sec:LWP}

In this section, we prove the local well-posedness of the Vlasov--Poisson system with the repulsive harmonic potential.  Despite the transport term $x\cdot \nabla_v \mu$ in \eqref{eq:VP-rep-potential-reformulate}, we do not require $\action{x}^3 \mu_0 \in \Leb{\infty}_{x,v}$ to show local well-posedness because the lens transform will eliminate this effect which can be seen in \eqref{eq:VP-lens} and $\gamma(0,q,p)=\mu(0,q,p)$.

\begin{theorem}\label{thm:lwp}
Let $\gamma_0 \in C^1_{q,p}$ satisfy 
\[
B:=\norm{\gamma_0}{\Leb{2}_{q,p}}+\norm{\action{p}^3\gamma_0}{\Leb{\infty}_{q,p}}+\norm{\nabla_{q,p}\gamma_0}{\Leb{\infty}_{q,p}}<\infty.
\]
Then there exist $S=S(B)\in(0,1)$ and an absolute constant $C>0$ such that the problem \eqref{eq:VP-lens} has a unique strong solution $\gamma \in C([0,S];\Leb{2}_{q,p})$ satisfying 
\[ \sup_{s\in [0,S]} \norm{\gamma(s)}{\Leb{2}_{q,p}}=\norm{\gamma_0}{\Leb{2}_{q,p}},\quad \sup_{s\in [0,S]} \norm{\action{p}^3\gamma(s)}{\Leb{\infty}_{q,p}}+\norm{\nabla_{q,p}\gamma(s)}{\Leb{\infty}_{q,p}}\leq CB.\]
\end{theorem}

\begin{proof}
The proof follows from standard Picard iteration.
Define the following iteration scheme:
\[ \phi_0(s,q)=0,\quad \gamma_{(0)}(s,q,p)=\gamma_0(q,p),\]
\begin{equation}\label{eq:picard-stage}
\partial_s \gamma_{(n+1)}+\{\gamma_{(n+1)},\mathcal{H}_n\}=0,\quad \gamma_{(n+1)}(s=0)=\gamma_0,\quad \mathcal{H}_n:=\frac{|p|^2}{2}-\lambda f(s)\phi_n(s,q),
\end{equation}
where 
\[ \phi_n(s,q)=\frac{1}{2\pi} \iint_{\mathbb{R}^2_y\times\mathbb{R}^2_{v}} \ln|q-y|\cdot{\gamma_{(n)}^2(s,y,v)}\myd{y}\myd{v}.\]
Similar to the proof of Lemma \ref{lem:estimate-E}, we have 
\begin{equation}\label{eq:electric-field-estimate-picard}
\begin{aligned}
\norm{D\phi_n}{\Leb{\infty}_q}&\apprle \norm{\gamma_{(n)}}{\Leb{2}_{q,p}}^2+\norm{\action{p}^2\gamma_{(n)}}{\Leb{\infty}_{q,p}}^2,\\
\norm{D^2\phi_n}{L^\infty_{q}}&\apprle \norm{\action{p}^2\gamma_{(n)}}{L^\infty_{q,p}}^2+\norm{\gamma_{(n)}}{L^2_{q,p}}^2+\norm{\nabla_q\gamma_{(n)}}{L^\infty_{q,p}}^2.
\end{aligned}
\end{equation}
Then by the method of characteristics, the problem admits a unique solution $\gamma_{(n+1)}\in C^0([0,1);\Sob{1}{\infty}_{q,p})\cap C^1([0,1);\Leb{\infty}_{q,p})$. Also, by differentiating \eqref{eq:picard-stage} in $q$ and $p$ (more precisely, taking a finite difference), we get 
\begin{align*}
\partial_s (D_{q^i} \gamma_{(n+1)})+\{D_{q^i}\gamma_{(n+1)},\mathcal{H}_n\}&=-\lambda f(s) \nabla_q(D_{q^i}\phi_n)\cdot\nabla_p\gamma_{(n+1)},\\
\partial_s (D_{p^i} \gamma_{(n+1)})+\{D_{p^i}\gamma_{(n+1)},\mathcal{H}_n\}&=-D_{q^i}\gamma_{(n+1)}.
\end{align*}
and for any weight $\omega$,
\begin{equation*}
	\begin{split}
	\pr_s(\omega\gamma_{(n+1)})+\{\omega\gamma_{(n+1)},\mathcal{H}_n\}=\gamma_{(n+1)}\{\omega,\mathcal{H}_n\}.	
	\end{split}
\end{equation*}

This implies that $\norm{\gamma_{(n+1)}}{L^2_{q,p}}=\norm{\gamma_0}{L^2_{q,p}}$
and 
\begin{equation}\label{eq:derivative-estimates}
\begin{aligned}
\norm{D_{q^i}\gamma_{(n+1)}(s)}{\Leb{r}_{q,p}}&\leq \norm{D_{q^i}\gamma_0}{\Leb{r}_{q,p}}+\int_0^s f(\tau)\norm{D^2\phi_n(\tau)}{\Leb{r}_q}\norm{\nabla_p \gamma_{(n+1)}(\tau)}{\Leb{r}_{q,p}}\myd{\tau},\\
\norm{D_{p^i}\gamma_{(n+1)}(s)}{\Leb{r}_{q,p}}&\leq \norm{D_{p^i}\gamma_0}{\Leb{r}_{q,p}}+\int_0^s \norm{\nabla_q \gamma_{(n+1)}(\tau)}{\Leb{r}_{q,p}}\myd{\tau},\\
\norm{\langle p\rangle^3\gamma_{(n+1)}(s)}{L^r_{q,p}}&\leq \norm{\langle p\rangle^3\gamma_0}{L^r_{q,p}}+\int_0^s \norm{\langle p\rangle^3\gamma_{(n+1)}(\tau)}{L^r_{q,p}}\norm{D\phi_n(\tau)}{L^\infty_q}f(\tau)\myd{\tau}.
\end{aligned}
\end{equation}
for $r\in\{2,\infty\}$.

Define
\begin{equation*}
\begin{aligned}
A_n(s)&:=\norm{\action{p}^3\gamma_{(n)}(s)}{\Leb{\infty}_{q,p}}+\norm{\nabla_{q,p}\gamma_{(n)}(s)}{ \Leb{\infty}_{q,p}}\\
B_n(s)&:=A_n(s)+\norm{\gamma_{(n)}(s)}{\Leb{2}_{q,p}},
\end{aligned}
\end{equation*}
Then by \eqref{eq:derivative-estimates}, we have
\[
	A_{n+1}(s)\leq A_0+\int_0^s (1+\norm{D\phi_n(\tau)}{L^\infty_q}+\norm{D^2\phi_n(\tau)}{L^\infty_q})A_{n+1}(\tau) f(\tau)\myd{\tau}.
\]
By Gronwall's inequality and \eqref{eq:electric-field-estimate-picard}, we have
\begin{equation*}
	\begin{split}
		A_{n+1}(s)&\apprle A_0\exp\left(\int_0^s (1+\norm{D\phi_n(\tau)}{L^\infty_q}+\norm{D^2\phi_n(\tau)}{L^\infty_q})f(\tau)\myd{\tau} \right)\\
		&\apprle A_0\exp\of{\int_0^s B_{n}^2(\tau)f(\tau)\myd{\tau}}.
	\end{split}
\end{equation*}
Then, by induction, there exist $C>1$ and $0<S(B)<1$,  such that 
\[ B_n(s)\leq CB,\quad \arctanh S(B)\sim B^{-2} \]
for all $s\in [0,S]$ and $n$.

If we write $\delta_n:=\gamma_{(n+1)}-\gamma_{(n)}$, then a direct computation gives
\[ 
0=\partial_s \delta_n +\left\{\delta_n,\frac{\mathcal{H}_n+\mathcal{H}_{n-1}}{2}\right\}+\lambda f(s) \nabla_q  (\phi_n-\phi_{n-1})\cdot \nabla_p \left(\frac{\gamma_{(n+1)}+\gamma_{(n)}}{2}\right).
\]
Note that 
\begin{align*}
\norm{\nabla_q [\phi_n-\phi_{n-1}]}{\Leb{\infty}_q}&\apprle \norm{\delta_{n-1}}{\Leb{\infty}_{q,p}}\left(\norm{\action{p}^3\gamma_{(n)}}{\Leb{\infty}_{q,p}}+\norm{\action{p}^3\gamma_{(n-1)}}{\Leb{\infty}_{q,p}} \right)\\
&\relphantom{=}+\norm{\delta_{n-1}}{\Leb{2}_{q,p}}(\norm{\gamma_{(n)}}{\Leb{2}_{q,p}}+\norm{\gamma_{(n-1)}}{\Leb{2}_{q,p}}).
\end{align*}

By Liouville's theorem induced by the new Hamiltonian $(\mathcal{H}_n+\mathcal{H}_{n-1})/2$, it follows that 
\begin{align*}
\frac{1}{2}\frac{d}{ds}\norm{\delta_n(s)}{\Leb{2}_{q,p}}^2&\apprle B^2\norm{\delta_{n-1}(s)}{\Leb{2}_{q,p}\cap\Leb{\infty}_{q,p}}\norm{\delta_n(s)}{\Leb{2}_{q,p}}f(s),
\end{align*}
and
\begin{align*}
\norm{\delta_n(s)}{\Leb{\infty}_{q,p}}&\apprle B^2 \int_0^s \norm{\delta_{n-1}(\tau)}{\Leb{2}_{q,p}\cap\Leb{\infty}_{q,p}}f(\tau)\,\mathrm{d}\tau.
\end{align*}
Hence by Gr\"{o}nwall's inequality, we get 
\[ \norm{\delta_n(s)}{\Leb{2}_{q,p}\cap\Leb{\infty}_{q,p}}\leq CB^2\int_0^s \norm{\delta_{n-1}(\tau)}{\Leb{2}_{q,p}\cap\Leb{\infty}_{q,p}} f(\tau) \,\mathrm{d}\tau \]
for some constant $C>0$. Then by iteration, we get 
\[ \norm{\delta_n(s)}{\Leb{2}_{q,p}\cap\Leb{\infty}_{q,p}}\leq (CB^2 \arctanh S)^n \]
for all $s\in [0,S]$ and for all $n$. Hence it follows that 
\[
 \norm{\gamma_{(m)}(s)-\gamma_{(n)}(s)}{\Leb{2}_{q,p}\cap\Leb{\infty}_{q,p}}\leq \sum_{k=n}^{m-1} (CB^2 \arctanh S)^k 
\]
for all $s\in [0,S]$. By choosing $S>0$ sufficiently small so that $CB^2 \arctanh S\leq 1/2$, we see that the sequence $\{\gamma_n\}$ is Cauchy in $\Leb{\infty}_s\Leb{2}_{q,p}\cap \Leb{\infty}_{s,q,p}$. Then by a standard compactness argument and duality argument, we can show the existence of a strong solution satisfying $\gamma \in C([0,S];\Leb{2}_{q,p})$ and $E[\gamma]=\lim_{n\rightarrow\infty} \nabla_q\phi_n$. The uniqueness part is also easy to show. Finally, the regularity result follows from a method of characteristics induced by the electric field generated by $\gamma$ and the uniqueness of the strong solution. This completes the proof of Theorem \ref{thm:lwp}.
\end{proof}

\begin{remark}
A similar proof also works for the Vlasov--Poisson with attractive harmonic potentials by using the trigonometric lens transformation \eqref{eq:VP-lens2}.
\end{remark}

\section{Scattering of solutions}\label{sec:scattering}
In this section, we prove Theorem \ref{thm:A}. The key point is to show that the electric field $E[\gamma](s)$ converges to $E_1$ uniformly as $s\rightarrow 1-$, which in particular implies $E[\gamma]$ is uniformly bounded. To show this, we need the following proposition, which tells us that the electric field is (almost) Lipschitz continuous assuming the bounds on the moments.
\begin{proposition}\label{prop:LipboundE}
Let $I\subset [0,1)$. If $\gamma$ satisfies 
\[ \gamma \in \Leb{\infty}(I;\Leb{\infty}_{q,p})\cap C(I;\Leb{2}_{q,p}),\quad |p|^2 \gamma \in \Leb{\infty}(I;\Leb{\infty}_{q,p}),\]
and 
\begin{equation}\label{eq:gamma-divergence-form}
	\partial_s \{\gamma^2\}+\Div_q(p\gamma^2)+\Div_p(G \gamma^2)=0\quad \text{in } I\times \mathbb{R}^2_q\times\mathbb{R}^2_p
\end{equation}
for some force field $G(s,q)$, then for $s_0,s_1 \in I$ satisfying $0\leq s_0\leq s_1<1$, we have
 \begin{align*}
 |E(s_1)-E(s_0)|&\apprle (s_1-s_0) |\ln (s_1-s_0)| \norm{\action{p}^2\gamma}{\Leb{\infty}_{s,q,p}}^2\\
 &\relphantom{=}+(s_1-s_0)^2[\norm{\gamma}{\Leb{\infty}_s\Leb{2}_{q,p}}^2+\norm{\action{p}^2\gamma}{\Leb{\infty}_{s,q,p}}^2]\\
 &\relphantom{=}+(s_1-s_0)^3\left(F(s_1)-F(s_0) \right)\norm{G/f(s)}{\Leb{\infty}_{s,q}} \norm{\action{p}^2\gamma}{\Leb{\infty}_{s,q,p}}^2.
 \end{align*}
\end{proposition}
\begin{proof}
Recall the representation of electric field $E$ by $E_R$ and $E_{R,V}$ given in \eqref{eq:electric-field-rep} and \eqref{eq:electric-field-rep-2}. Fix $A,B\in (0,1)$ which will be determined later. Then by Lemma \ref{lem:estimate-E_RV}, we have
\[
 	\int_{R=A^{-1}}^\infty |E_R(s)|\frac{\myd{R}}{R^2}\apprle A\norm{\gamma}{L^2_{q,p}}^2, \quad \int_0^A |E_R(s)|\frac{\myd{R}}{R^2}\apprle A\norm{\action{p}^2\gamma}{L^\infty_{q,p}}^2.
\]
 and
\[
 		\int_{R=0} ^{A^{-1}}\of{\int_{V=0}^B+\int_{B^{-1}}^\infty} |E_{R,V}|\frac{\myd{R}}{R^2}\frac{\myd{V}}{V}\apprle A^{-1}B^2(\norm{\gamma}{L^\infty_{q,p}}^2+\norm{|{p}|^2\gamma}{L^\infty_{q,p}}^2).
\]

From these, we get 
 \begin{equation}\label{eq:remainder-estimates}
 	\begin{split}
 		\Bigl| E(s)-\int_{A}^{A^{-1}} \mathcal{E}_{R,B}(s)\frac{\myd{R}}{R^2}\Bigr|\apprle A^{-1}B^2\norm{\action{p}^2\gamma}{L^{\infty}_{q,p}}^2+A[\norm{\gamma}{L^2_{q,p}}^2+ \norm{\action{p}^2\gamma}{L^{\infty}_{q,p}}^2], \\
 		\mathcal{E}_{R,B}(s):= \iint_{\mathbb{R}^2_y\times\mathbb{R}^2_u} (\nabla\chi)(R^{-1}(x-y))\cdot \widetilde{\chi}_{\{B\leq\cdot\leq B^{-1}\}}(u)\cdot \gamma^2(s,y,u)\,\mathrm{d}y\,\mathrm{d}u
 	\end{split}
 \end{equation}
 where
 \begin{equation*}
 	\begin{split}
 		\widetilde{\chi}_{\{B\leq\cdot\leq B^{-1}\}}(u):=\int_{\{B\leq V\leq B^{-1}\}} \widetilde{\chi}(V^{-1}u)\frac{\myd{V}}{V}.
 	\end{split}
 \end{equation*}
 Using the equation \eqref{eq:gamma-divergence-form}, for any $s_0,s_1\in I$ with $s_0<s_1$, we  get
 \begin{equation}\label{eq:difference-estimate}
\begin{aligned}
0&=  \int_{s=s_0}^{s_1} \iint_{\mathbb{R}^2_r\times\mathbb{R}^2_u} (\nabla\chi)\left(R^{-1}(q-r)\right) \cdot \widetilde{\chi}_{\left\{B \leq \cdot \leq B^{-1}\right\}}(u)\\
&\relphantom{=}\times \left\{\partial_s \gamma^2+\operatorname{div}_r\left(\gamma^2 u\right)+\operatorname{div}_u\left(G \gamma^2\right)\right\} \,\mathrm{d} r \,\mathrm{d} u\,\mathrm{d} s, \\
&=  \mathcal{E}_{R, B}\left(s_1\right)-\mathcal{E}_{R, B}\left(s_0\right)\\
&\relphantom{=}+\int_{s=s_0}^{s_1} \iint_{\mathbb{R}^2_r\times\mathbb{R}^2_u} R^{-1} u^k\left\{\nabla \partial_{q^k} \chi\right\}\left(R^{-1}(q-r)\right)\\
&\relphantom{=}\times  \widetilde{\chi}_{\left\{B \leq \cdot \leq B^{-1}\right\}}(u) \cdot \gamma^2(s, r, u) \,\mathrm{d} r \,\mathrm{d} u\,\mathrm{d} s\\
&\relphantom{=} -\int_{s=s_0}^{s_1} \iint_{\mathbb{R}^2_r\times\mathbb{R}^2_u}  \partial_{q^j} \chi\left(R^{-1}(q-r)\right) \cdot \gamma^2(s, r, u) \cdot\left(G \cdot \nabla_u\right) \widetilde{\chi}_{\left\{B \leq \cdot \leq B^{-1}\right\}}(u) \,\mathrm{d} r \,\mathrm{d} u \,\mathrm{d} s.
\end{aligned}
\end{equation}
Here we used the Einstein summation convention for the repeated indices. 

Since
$$
\left|\nabla_u \widetilde{\chi}_{\left\{B \leq \cdot \leq B^{-1}\right\}}(u)\right| \apprle B^{-1} 1_{\{|u| \leq 2 B\}}+B 1_{\left\{|u| \geq B^{-1} / 2\right\}},
$$
we see that
$$
\begin{aligned}
&\relphantom{=} \left|\iint_{\mathbb{R}^2_r\times\mathbb{R}^2_u}  \partial_{q^j} \chi\left(R^{-1}(q-r)\right) \cdot \gamma^2(r, u) \cdot\left(G \cdot \nabla_u\right) \widetilde{\chi}_{\left\{B \leq \cdot \leq B^{-1}\right\}}(u) \,\mathrm{d}r \,\mathrm{d}u\right| \\
& \apprle\|G/f(s)\|_{L_{s, q}^{\infty}} \cdot f(s)\cdot R^2 \cdot\left[B\|\gamma\|_{L_{q, p}^{\infty}}^2+B^2\left\||p|^2 \gamma\right\|_{L_{q, p}^{\infty}}^2\right].
\end{aligned}
$$

Then by \eqref{eq:difference-estimate}, we get 
\begin{equation}\label{eq:almost-lipschitz}
 \begin{aligned}
 \left|\int_{A}^{A^{-1}} \{\mathcal{E}_{R,B}(s_1)-\mathcal{E}_{R,B}(s_0)\}\frac{\myd{R}}{R^2} \right|&\apprle (s_1-s_0) \norm{\action{p}^2\gamma}{\Leb{\infty}_{s,q,p}}^2\int_{R=A}^{A^{-1}} \frac{\myd{R}}{R} \\
 &\relphantom{=}+\left(\int_{s_0}^{s_1} f(s)ds \right)\norm{G/f}{\Leb{\infty}_{s,q}}A^{-1}B \norm{\action{p}^2\gamma}{\Leb{\infty}_{s,q,p}}^2.
 \end{aligned}
 \end{equation}
Hence it follows from \eqref{eq:remainder-estimates} and \eqref{eq:almost-lipschitz} and that
 \begin{align*}
 |E(s_1)-E(s_0)|&\apprle A[\norm{\gamma}{\Leb{2}_{q,p}}^2+\norm{\action{p}^2\gamma(s)}{\Leb{\infty}_{q,p}}^2]+A^{-1}B^2\norm{\action{p}^2\gamma(s)}{\Leb{\infty}_{q,p}}^2\\
 &\relphantom{=}+(s_1-s_0)\norm{\action{p}^2\gamma}{\Leb{\infty}_{s,q,p}}^2 (-\ln A)\\
 &\relphantom{=}+\left(F(s_1)-F(s_0) \right)\norm{G/f}{\Leb{\infty}_{s,q}}A^{-1}B \norm{\action{p}^2\gamma}{\Leb{\infty}_{s,q,p}}^2.
 \end{align*}
 If we choose $A=(s_1-s_0)^2$ and $B=(s_1-s_0)^5$, then 
 \begin{align*}
 |E(s_1)-E(s_0)|&\apprle (s_1-s_0)^2[\norm{\gamma}{\Leb{\infty}_s\Leb{2}_{q,p}}^2+\norm{\action{p}^2\gamma}{\Leb{\infty}_{s,q,p}}^2]\\
 &\relphantom{=}+(s_1-s_0) |\ln (s_1-s_0)|\norm{\action{p}^2\gamma}{\Leb{\infty}_{s,q,p}}^2\\
 &\relphantom{=}+(s_1-s_0)^3\left(F(s_1)-F(s_0)\right)\norm{G/f(s)}{\Leb{\infty}_{s,q}} \norm{\action{p}^2\gamma}{\Leb{\infty}_{s,q,p}}^2.
 \end{align*}
 This proves Proposition \ref{prop:LipboundE}.
\end{proof}

The next lemma shows that we can propagate the moments assuming the electric field is bounded.
\begin{lemma}\label{lemma:propagatmoment} Fix an integer $a>0$. Let $\gamma$ be a solution to \eqref{eq:Hamiltonian-gamma} on $[0,T^*]$ with initial data $\gamma(s=0)=\gamma_0$. Suppose that $\gamma_0$ satisfies 
\[
 \norm{\action{p}^a \gamma_0}{\Leb{r}_{q,p}}\leq \varepsilon_0
\]
and that 
\[
|E(s,q)|\leq D,\quad 0\leq s\leq T^*.
\]
Then 
\begin{align*}
\norm{\gamma(s)}{\Leb{r}_{q,p}}&\leq \varepsilon_0,\\
\norm{\action{p}^a\gamma(s)}{\Leb{r}_{q,p}}&\apprle \varepsilon_0+\varepsilon_0 \action{F(s)}^a D^a.
\end{align*}
for all $s\in [0,T^*]$.
\end{lemma}
\begin{proof}
A direct computation gives
\begin{equation*}
	\begin{split}
	\pr_s(\action{p}^a\gamma)+\{\action{p}^a \gamma,\H \}=\lambda a f(s)  \action{p}^{a-2}p\gamma\cdot E(s,q).	
	\end{split}
\end{equation*}
Then the proof follows from Lemma \ref{lem:Liouville} and induction.
\end{proof}

Finally, we note the following commutative relation of the derivatives: if $\gamma$ is a solution to \eqref{eq:Hamiltonian-gamma}, then for $1\leq i,j\leq 2$, we have
\begin{equation}\label{eq:derivative-relation-gamma}  
\begin{aligned}
\partial_s(\pr_{q^i}\gamma)+\{\pr_{q^i}\gamma,\mathcal{H}\}&=-\lambda f(s)\pr_{q^i} E \cdot \nabla_p \gamma, \\
\partial_s(\pr_{p^i}\gamma)+\{\pr_{p^i}\gamma,\mathcal{H}\}&=-\pr_{q^i}\gamma,\\
\partial_s(p^i\pr_{p^j}\gamma)+\{p^i\pr_{p^j}\gamma,\mathcal{H}\}&=-p^i\pr_{q^j}\gamma+\lambda f(s)E^i\cdot \pr_{p^j}\gamma,    \\
\partial_s(p^i\pr_{q^j}\gamma)+\{p^i\pr_{q^j}\gamma,\mathcal{H}\}&=-\lambda f(s)p^i\pr_{q^j}E\cdot \nabla_p\gamma+\lambda f(s)E^i\pr_{q^j}\gamma.
\end{aligned}
\end{equation}

The following bootstrap proposition will be used to show that the local-in-time solution becomes global under a smallness assumption on the initial data.

\begin{proposition}\label{prop:bootstrap}
There exists a small $\varepsilon_1>0$ such that if $\gamma$ is a solution of \eqref{eq:Hamiltonian-gamma} on $[0,T^*]$ with initial data $\gamma(s=0)=\gamma_0$, then we have the following bootstrap results: 
\begin{enumerate}[\rm (i)\,]
\item (Moments and the electric field) if there holds that 
\begin{equation}\label{eq:initial-gamma}
	\norm{\action{p}\gamma_0}{\Leb{2}_{q,p}}+\norm{\action{p}^m \gamma_0}{\Leb{\infty}_{q,p}}\leq \varepsilon_1,\quad m\geq 2,
\end{equation}
then the electric field $E(s)$ remains bounded and the solution $\gamma$ satisfies the bounds 
\begin{equation}\label{eq:p-gamma-L2}
\norm{\action{p}\gamma(s)}{\Leb{2}_{q,p}}\apprle \varepsilon_1 \action{F(s)}
\end{equation}
and 
\begin{equation}\label{eq:p-gamma-L-infinity}
\norm{\action{p}^a\gamma(s)}{\Leb{\infty}_{q,p}}\apprle \varepsilon_1 \action{F(s)}^a,\quad 0\leq a\leq m.
\end{equation}
Moreover, there exists $C>0$ such that 
\[ |E(s_1,q)-E(s_0,q)|\leq C \varepsilon_1^2(s_1-s_0)\action{F(s_1)}^{4}|\ln(s_1-s_0)| \]
for any $0\leq s_0< s_1\leq T^{*}$.
\item (Derivatives) Assume additionally that for some $b\in \{0,1\}$, there holds that 
\begin{equation}\label{eq:initial-dgamma}
	\norm{\action{p}^b\nabla_{q,p} \gamma_0}{\Leb{\infty}_{q,p}}\leq \varepsilon_1.
\end{equation}
Then we have the bounds
\begin{equation}\label{eq:derivative-control-estimates}
\begin{aligned}
\norm{\action{p}^a \nabla_p \gamma(s)}{\Leb{\infty}_{q,p}}&\apprle \varepsilon_1\action{F(s)}^a,\quad 0\leq a\leq b,\\
\norm{\action{p}^a \nabla_q \gamma(s)}{\Leb{\infty}_{q,p}}&\apprle \varepsilon_1\action{F(s)}^{5+a},\quad 0\leq a\leq b.
\end{aligned}
\end{equation}
\end{enumerate}
\end{proposition}
\begin{proof}
(i) Let $C>2$ be a constant larger than all the implied constants that appear in Lemma \ref{lem:estimate-E}, Theorem \ref{thm:lwp}, and Proposition \ref{prop:LipboundE}. Let $\varepsilon_1>0$ be small enough so that 
\[ 4C_*^2 \varepsilon_1^2\leq 1,\]
where $C_*\geq C$ and $\varepsilon_1>0$ will be determined later.

Define 
\[ I=\{s \in [0,T^*] : \norm{E(s)}{\Leb{\infty}_q}\leq 2C_*^2\varepsilon_1^2\}.\]
By \eqref{eq:boundE} with $A=1$ and \eqref{eq:initial-gamma}, we get 
\[
\norm{E(0)}{\Leb{\infty}_q}\leq C\left(\norm{\gamma_0}{\Leb{2}_{q,p}}^2+\norm{\gamma_0}{\Leb{\infty}_{q,p}}^2+\norm{|p|^2\gamma_0}{\Leb{\infty}_{q,p}}^2\right)\leq 2C\varepsilon_1^2,
\]
which implies $0\in I$. By continuity, it suffices to show that $I$ is open. 

By Lemma \ref{lemma:propagatmoment}, $0\leq a\leq m$, for $s \in I$, and $r\in \{2,\infty\}$, we have
\begin{equation}\label{eq:momentum-s-interval}
\norm{\action{p}^a \gamma(s)}{\Leb{r}_{q,p}}\apprle \varepsilon_1+\varepsilon_1\action{F(s)}^a(C_*^2\varepsilon_1^2)^a.
\end{equation}
Then by \eqref{eq:momentum-s-interval} and Proposition \ref{prop:LipboundE}, for $s_1,s_0\in I$ satisfying $0\leq s_0<s_1<1$, we have 
\begin{equation}\label{eq:electric-field-lipschitz-dyadic}
\begin{aligned}
 |E(s_1)-E(s_0)|&\apprle (s_1-s_0) |\ln (s_1-s_0)| \norm{\action{p}^2\gamma}{L^\infty_s([s_0,s_1],L^\infty_{q,p})}^2\\
 &\relphantom{=}+(s_1-s_0)^2[\norm{\gamma}{\Leb{2}_{q,p}}^2+\norm{\action{p}^2\gamma}{L^\infty_s([s_0,s_1],L^\infty_{q,p})}^2]\\
 &\relphantom{=}+(s_1-s_0)^3\left(F(s_1)-F(s_0) \right)\norm{E}{\Leb{\infty}_{s,q}} \norm{\action{p}^2\gamma}{L^\infty_s([s_0,s_1],L^\infty_{q,p})}^2\\
 &\apprle \varepsilon_1^2(s_1-s_0)|\ln(s_1-s_0)|\action{F(s_1)}^4\\
 &\relphantom{=}+\varepsilon_1^2(s_1-s_0)^2\action{F(s_1)}^4\\
 &\relphantom{=}+\varepsilon_1^2(s_1-s_0)^3(F(s_1)-F(s_0))\action{F(s_1)}^4.
 \end{aligned}
 \end{equation}
 
 Define $r_0=0$, $r_k=\sum_{j=1}^k 2^{-j}$ for $k\geq 1$. Then $r_{k+1}-r_k=2^{-k-1}$ and $r_k\rightarrow 1$ as $k\rightarrow\infty$. By the mean value theorem, there exists $r_k^*\in (r_k,r_{k+1})$ such that 
\begin{align*}
F(r_{k+1})-F(r_k)&=(r_{k+1}-r_k)\frac{1}{1-(r_k^*)^2}\\
&\leq \frac{1}{1-r_{k+1}}(r_{k+1}-r_k)\leq 2^{k+1}(r_{k+1}-r_k)\leq 1
\end{align*}
and
\begin{align*}
\action{F(r_{k+1})}&\apprle k.
\end{align*}
Hence it follows that for $r_k\leq s_0<s_1\leq r_{k+1}$, we have 
 \begin{equation}\label{eq:convergence-rate}
 \norm{E(s_1)-E(s_0)}{\Leb{\infty}_q}\apprle \varepsilon_1^2 \left(\frac{k^5}{2^k}+\frac{k^4}{2^{2k}}+\frac{k^4}{2^{3k}}\right).
 \end{equation}
 
 Since $s\in I$, it follows from \eqref{eq:convergence-rate} that
 \begin{equation}\label{eq:electric-field-boundedness}
 \begin{aligned}
 \norm{E(s)}{\Leb{\infty}_q}&\leq \norm{E(0)}{\Leb{\infty}_q}+C_2\varepsilon_1^2 \sum_{k=1}^\infty  \left(\frac{k^5}{2^k}+\frac{k^4}{2^{2k}}+\frac{k^4}{2^{3k}}\right)\\
 &\leq (2C+C_3)\varepsilon_1^2.
 \end{aligned}
 \end{equation}
 Hence by choosing $C_*\geq 2C+C_3$ and the continuity of $E$, we get $I$ is open.
 
 To show (ii), we use a similar bootstrap argument. We assume the bootstrap assumption  below:
 \begin{equation}\label{eq:bootstrap-assumption-derivatives}
 \begin{aligned}
 \norm{\action{p}^a\nabla_p \gamma(s)}{\Leb{\infty}_{q,p}}&\leq 2C^4 \varepsilon_1 \action{F(s)}^a,\quad 0\leq a\leq b,\\
 \norm{\action{p}^a\nabla_q\gamma(s)}{\Leb{\infty}_{q,p}}&\leq 2C^2 \varepsilon_1 \action{F(s)}^{5+a},\quad 0\leq a\leq b.
 \end{aligned}
 \end{equation}
 where $\varepsilon_1\ll 1$ to be chosen later. The main point of the bootstrap argument is that the final estimate yields a factor of $\varepsilon_1^3$, which is sufficient to absorb all the constants.

 By \eqref{eq:derivative-relation-gamma} and the bootstrap assumption \eqref{eq:bootstrap-assumption-derivatives}, we have 
\[
\norm{\nabla_p\gamma(s)}{\Leb{\infty}_{q,p}}\leq \norm{\nabla_p \gamma(0)}{\Leb{\infty}_{q,p}}+\int_0^s \norm{\nabla_q\gamma(\tau)}{\Leb{\infty}_{q,p}}\myd{\tau}\leq C^3 2\varepsilon_1.
\]
and 
 \begin{equation}\label{eq:gradient-gamma-q-bootstrap-1}
 \begin{aligned}
 \norm{\nabla_q \gamma(s)}{\Leb{\infty}_{q,p}}&\leq \norm{\nabla_q\gamma(0)}{\Leb{\infty}_{q,p}}+\int_0^s f(\tau)\norm{\nabla_q E}{\Leb{\infty}_q}\norm{\nabla_p \gamma}{\Leb{\infty}_{q,p}}\myd{\tau}\\
 &\leq \varepsilon_1+\varepsilon_1C^4\int_0^s f(\tau)\norm{\nabla_q E}{\Leb{\infty}_q}\myd{\tau}.
 \end{aligned}
 \end{equation}
 
 By \eqref{eq:gradient-E-control}, \eqref{eq:p-gamma-L-infinity}, and \eqref{eq:bootstrap-assumption-derivatives}, we have
 \begin{equation}\label{eq:gradient-E-bound-bootstrap}
 \begin{aligned}
 &\relphantom{=}\norm{\nabla_q E(s)}{\Leb{\infty}_q}\\
 &\apprle \action{F(s)}^4\norm{\gamma}{\Leb{2}_{q,p}}^2+\action{F(s)}^{-6/5}\norm{\gamma}{\Leb{\infty}_{q,p}}\norm{\nabla_q\gamma}{\Leb{\infty}_{q,p}}+\action{F(s)}^{-1/5}\norm{|p|^2\gamma}{\Leb{\infty}_{q,p}}^2\\
 &\apprle \varepsilon_1^2\action{F(s)}^4 +\varepsilon_1 \action{F(s)}^{-6/5}\norm{\nabla_q\gamma}{\Leb{\infty}_{q,p}}+\varepsilon_1^2 \action{F(s)}^{4-1/5}.
 \end{aligned}
 \end{equation}
Then it follows from \eqref{eq:gradient-gamma-q-bootstrap-1} and \eqref{eq:gradient-E-bound-bootstrap} that 
 \begin{align*}
  \norm{\nabla_q \gamma(s)}{\Leb{\infty}_{q,p}}&\apprle \varepsilon_1+\varepsilon_1^3 C^4\int_0^s f(\tau) F(\tau)^4 \myd{\tau}+\varepsilon_1^3 C^4\int_0^s f(\tau)\action{F(\tau)}^{4-1/5}\myd{\tau}\\
  &\relphantom{=}+\varepsilon_1^2 C^4\int_0^s \action{F(\tau)}^{-6/5}f(\tau)\norm{\nabla_q\gamma}{\Leb{\infty}_{q,p}}\myd{\tau}\\
  &\apprle \varepsilon_1+\varepsilon_1^3C^4 \action{F(s)}^5+\varepsilon_1^2 C^4\int_0^s \action{F(\tau)}^{-6/5}f(\tau)\norm{\nabla_q\gamma}{\Leb{\infty}_{q,p}}\myd{\tau}.
 \end{align*}
Hence it follows from Gr\"{o}nwall's inequality and choosing $\varepsilon_1$ small enough that 
\begin{equation}\label{eq:gradient-q-bound}
 \norm{\nabla_q \gamma(s)}{\Leb{\infty}_{q,p}}\leq [\varepsilon_1+\varepsilon_1^3 C_0C^4 \action{F(s)}^5]\exp(\varepsilon_1^2C^4C_0)\leq C^2\varepsilon_1 \action{F(s)}^5.
\end{equation}
Then by \eqref{eq:initial-dgamma},  \eqref{eq:electric-field-boundedness}, \eqref{eq:bootstrap-assumption-derivatives}, and \eqref{eq:gradient-gamma-q-bootstrap-1}, we have 
\begin{align*}
\norm{p^j \partial_{p^i}\gamma(s)}{\Leb{\infty}_{q,p}}&\leq \norm{|p|\nabla_p \gamma(0)}{\Leb{\infty}_{q,p}}\\
&\relphantom{=}+\int_0^s \norm{|p|\nabla_q\gamma(\tau)}{\Leb{\infty}_{q,p}}\myd{\tau}+\int_0^s\norm{E(\tau)}{\Leb{\infty}_q}\norm{\nabla_p \gamma(\tau)}{\Leb{\infty}_{q,p}}f(\tau)\myd{\tau}\\
&\leq \varepsilon_1+2C^2\varepsilon_1 C_0+C_*\varepsilon_1^3 C^4 C_0 \action{F(s)}\leq C^4\varepsilon_1 \action{F(s)},
\end{align*}
where we choose $\varepsilon_1$ small enough in the last inequality.
To obtain the last estimate, it follows from \eqref{eq:gradient-E-bound-bootstrap} and \eqref{eq:gradient-q-bound} that
\begin{align*}
\norm{\nabla_q E(s)}{\Leb{\infty}_{q}}\apprle  \varepsilon_1^2C^4 \action{F(s)}^4,
\end{align*}
which implies that 
\begin{align*}
\norm{p^j\nabla_q\gamma(s)}{\Leb{\infty}_{q,p}}&\leq \norm{|p|\nabla_q\gamma_0}{\Leb{\infty}_{q,p}}\\
&\relphantom{=}+\int_0^s \left(\norm{\nabla_q E}{\Leb{\infty}_q}\norm{|p|\nabla_p\gamma}{\Leb{\infty}_{q,p}}+\norm{E}{\Leb{\infty}_q}\norm{\nabla_q\gamma}{\Leb{\infty}_{q,p}}\right)f(\tau)\myd{\tau}\\
&\leq \varepsilon_1+\varepsilon_1^3 C^8 C_0 \int_0^s \action{F(\tau)}^5f(\tau)\myd{\tau}\\
&\leq \varepsilon_1+C_0C^8\varepsilon_1^3 \action{F(s)}^6.
\end{align*}
Then the proof follows by choosing again $\varepsilon_1$ small enough.
\end{proof}

Now we are ready to prove Theorem \ref{thm:A}.
\begin{proof}[Proof of Theorem \ref{thm:A}]
First, we note that \eqref{eq:initial-mu} leads to \eqref{eq:initial-gamma} and \eqref{eq:initial-dgamma} by the lens transform. By Theorem \ref{thm:lwp}, there exists a local-in-time strong solution $\gamma$ of \eqref{eq:VP-lens} on $[0,S)$. To show that the solution exists globally, recall \eqref{eq:electric-field-lipschitz-dyadic} that 
\begin{equation}\label{eq:lipschitz-recall-inequality}
\begin{aligned}
\norm{E(s_1)-E(s_0)}{\Leb{\infty}_q}&\apprle\varepsilon_0^2(s_1-s_0)|\ln(s_1-s_0)|\action{F(s_1)}^4\\
 &\relphantom{=}+\varepsilon_0^2(s_1-s_0)^2\action{F(s_1)}^4\\
 &\relphantom{=}+\varepsilon_0^2(s_1-s_0)^3(F(s_1)-F(s_0))\action{F(s_1)}^4
 \end{aligned}
\end{equation}
for $0\leq s_0\leq s_1<1$.

Let $r_k=1-1/2^k$. Then by \eqref{eq:convergence-rate}, $\{E(r_k)\}$ forms a Cauchy sequence in $\Leb{\infty}_q$ and hence the limit $\lim_{k\rightarrow\infty} E(r_k,q)=E_1(q)$ exists. For $s\in (0,1)$, choose $k$ so that $r_k< s\leq r_{k+1}$. Then by \eqref{eq:lipschitz-recall-inequality}, we have  
\begin{equation}\label{eq:convergence-rate-E1}
\begin{aligned}
\norm{E_1-E(s)}{\Leb{\infty}_q}&\leq \norm{E_1-E(r_k)}{\Leb{\infty}_q}+\norm{E(r_k)-E(s)}{\Leb{\infty}_q}\\
&\apprle \varepsilon_0^2 \sum_{l=k}^\infty \left(\frac{l^5}{2^l}+\frac{l^4}{2^{2l}}+\frac{l^4}{2^{3l}}\right)\\
&\relphantom{=}+\varepsilon_0^2(s-r_k)|\ln(s-r_k)|\action{F(s)}^4+\varepsilon_0^2(s-r_k)^2\action{F(s)}^4\\
&\relphantom{=}+\varepsilon_0^2(s-r_k)^3(F(s)-F(r_k))\action{F(s)}^4\\
&\apprle \varepsilon_0^2 (1-s)\action{\ln (1-s)}^5.
\end{aligned}
\end{equation}
Hence, the electric field remains bounded, and by Proposition \ref{prop:bootstrap}, the unique local solution becomes a global solution of \eqref{eq:VP-lens} in $[0,1)$ and $\norm{E_1}{\Leb{\infty}_q}\apprle \varepsilon_0^2$. Moreover, it follows from \eqref{eq:bound-dE} and Proposition \ref{prop:bootstrap} that
\begin{equation}\label{eq:gradient-growth-electric-fields}
 \norm{\nabla E(s)}{L^\infty_q}\apprle \varepsilon_0^2 \action{F(s)}^4	.	
\end{equation}

To describe the asymptotic behavior of solutions, define 
\begin{equation}\label{eq:nu}
	\Phi(s,q,p):=(\Phi^1(s,q,p),\Phi^2(s,q,p))=(q+(s-1)p+\lambda(s-1)F(s)E_1(q),p+\lambda F(s)E_1(q))
\end{equation}
and 
\[ \nu(s,q,p):=\gamma(s,\Phi(s,q,p)).\]
A direct computation gives
\begin{equation}\label{eq:nu-estimates}
\begin{aligned}
&\relphantom{=}\partial_s \nu(s,q,p)\\
&=\partial_s \gamma+[p+\lambda F(s)E_1(q)+\lambda(s-1)F'(s)E_1(q)]\cdot \nabla_q \gamma + \lambda F'(s)E_1(q)\cdot\nabla_p\gamma\\
&=\lambda[(s-1)f(s)E_1(q)]\cdot \nabla_q \gamma+\lambda f(s)[E_1(q)-E(s,\Phi^1(s,q,p))]\cdot\nabla_p\gamma 
\end{aligned}
\end{equation}

By the mean value theorem, \eqref{eq:convergence-rate-E1}, and \eqref{eq:gradient-growth-electric-fields}, we have
\begin{equation*}
	\begin{split}
|E_1(q)-E(s,\Phi^1(s,q,p))|&	\leq |E_1(q)-E(s,q)|+|E(s,q)-E(s,\Phi^1(s,q,p))| \\	
&\apprle \varepsilon_0^2 (1-s)\action{\ln(1-s)}^5+\varepsilon_0^2\action{F(s)}^4(1-s)|\Phi^2(s,q,p)|.
	\end{split}
\end{equation*}

Hence it follows from \eqref{eq:bootstrap-assumption-derivatives} and \eqref{eq:nu-estimates} that 
\begin{equation}\label{bound:ds-nu}
	\begin{aligned}
\norm{\partial_s \nu}{\Leb{\infty}_{q,p}}&\apprle \varepsilon_0^3 \action{F(s)}^5+\varepsilon_0^3\action{\ln(1-s)}^5+\varepsilon_0^3\action{F(s)}^{10}\\
&\apprle \varepsilon_0^3 \action{\ln(1-s)}^{10}.
	\end{aligned}
\end{equation}
which is integrable near $s=1$. This proves that 
\[ \lim_{s\rightarrow 1-} \nu(s,q,p)=\mu_\infty(q,p) \quad \hbox{in }  L^\infty_{q,p}  \]
with $E_{1}=E_{\infty}$. Applying the inverse of the lens transform, we obtain \eqref{eq:scattering} and \eqref{eq:scattering-trajectory}. Finally, it remains to show that $E_{\infty}=E[\mu_{\infty}]$, which is 
	\begin{equation*}
		\begin{aligned}
E[\gamma](s,x)&=\iint_{\mathbb{R}^2_q\times\mathbb{R}^2_p} \frac{x-q}{|x-q|^2} \gamma^2(s,q,p)\myd{q}\myd{p}\\
&\rightarrow E[\mu_{\infty}](x)=\iint_{\mathbb{R}^2_q\times\mathbb{R}^2_p} \frac{x-q}{|x-q|^2} \mu_{\infty}^2(q,p)\myd{q}\myd{p}	
		\end{aligned}
	\end{equation*}
	as $s\rightarrow 1-$.
Using the change of variable \eqref{eq:nu}, we have
\begin{equation*}
	\begin{split}
E[\gamma](s,x)=\iint_{\mathbb{R}^2_q\times\mathbb{R}^2_p} \frac{x-(q+(s-1)G(s,q,p))}{|x-(q+(s-1)G(s,q,p))|^2}\nu^2(s,q,p)\myd{q}\myd{p}, 		
	\end{split}
\end{equation*}	
where
\begin{equation*}
	\begin{split}
	G(s,q,p)=p+\lambda F(s)E_1(q).	
	\end{split}
\end{equation*}	
By Fatou's lemma and conservation, we have $\norm{\mu_{\infty}}{L^2_{q,p}}\apprle \varepsilon_0$ and $\norm{\nu}{L^2_{q,p}}\apprle \varepsilon_0$. Moreover,  we have $\norm{\langle p\rangle^3\nu(s)}{L^{\infty}_{q,p}}\les \varepsilon_0 \langle \ln (1-s)\rangle^3$, which implies 
\begin{equation}\label{eq:sup-nu-decay}
\sup_{q}\int_{\mathbb{R}^2} |\nu(s,q,p)|dp\les \varepsilon_0  \langle\ln (1-s)\rangle^3.
\end{equation}
Define $s_n=1-(1-s)2^{-n}$ and $\nu_n(q,p)=\nu(s_n,q,p)$. Then by \eqref{bound:ds-nu} and \eqref{eq:sup-nu-decay}, we have
\begin{equation*}
	\begin{split}
		\sup_{q}\int_{\mathbb{R}^2} |\nu_{n+1}^2(q,p)-\nu_{n}^2(q,p)|dp\les \varepsilon_0^4 (1-s)2^{-n}(\langle \ln (1-s)\rangle+n)^8.
	\end{split}
\end{equation*}
By taking the summation and applying Fatou's lemma, we get 
\begin{equation}\label{eq:convergence-rate-L1}
		\sup_{q} \int_{\mathbb{R}^2} |\nu(s,q,p)^2-\mu_{\infty}^2(q,p)|dp\les \varepsilon_0^4(1-s)\langle \ln (1-s)\rangle^8.
\end{equation}
 Therefore, by 	\eqref{eq:convergence-rate-L1}, we have
\begin{equation}\label{eq:rep-formula-nu-infty}
	\begin{aligned}
&\relphantom{=}\left| \iint_{\mathbb{R}^2_q\times\mathbb{R}^2_p} \frac{x-q}{|x-q|^2}[\nu^2(s,q,p)-\mu^2_{\infty}(q,p)]\myd{q}\myd{p}     \right|	\\
&\approx \left|\int_0^\infty \iint_{\mathbb{R}^2_q\times\mathbb{R}^2_p} (\nabla\chi)(R^{-1}(x-q))[\nu^2(s,q,p)-\mu^2_{\infty}(q,p)]\myd{q}\myd{p}\frac{\myd{R}}{R^2} \right|	\\
&\apprle \int_0^\infty \min\{R^{-2}(\norm{\nu}{L^2_{q,p}}^2+\norm{\mu_{\infty}}{L^2_{q,p}}^2),\, (1-s)^{0.9}  \} \myd{R} \, \apprle (1-s)^{0.45}\longrightarrow 0
	\end{aligned}
\end{equation}
as $s\rightarrow 1-$. 

It remains to show that
\begin{equation*}
	\begin{split}
\lim_{s\rightarrow 1}\iint_{\mathbb{R}^2_q\times\mathbb{R}^2_p} \off{\frac{x-(q+(s-1)G(s,q,p))}{|x-(q+(s-1)G(s,q,p))|^2}-\frac{x-q}{|x-q|^2}}\nu^2(s,q,p)\myd{q}\myd{p}		=0.
	\end{split}
\end{equation*}	

We change the coordinates back to $\gamma$ and use \eqref{eq:electric-field-rep} and the change of variable $Q=q+(s-1)G(s,q,p),\, P=G(s,q,p)$ to get
\begin{align*}
		&\relphantom{=}\left|\iint_{\mathbb{R}^2_q\times\mathbb{R}^2_p} \off{\frac{x-(q+(s-1)G(s,q,p))}{|x-(q+(s-1)G(s,q,p))|^2}-\frac{x-q}{|x-q|^2}}\nu^2(s,q,p)\myd{q}\myd{p}		\right| \\
		&= \left|
\iint_{\mathbb{R}^2_Q\times\mathbb{R}^2_P}
\left[\frac{x-Q}{|x-Q|^2}-\frac{x-Q+(s-1)P}{|x-Q+(s-1)P|^2}\right]\gamma^2(s,Q,P)\,\mathrm{d}Q\,\mathrm{d}P\right| \\
   &= \left|\iint_{\mathbb{R}^2_q\times\mathbb{R}^2_p} \off{\frac{x-q}{|x-q|^2}-\frac{x-q+p(s-1)}{|x-q+p(s-1)|^2}}\gamma^2(s,q,p)\myd{q}\myd{p}		\right| 
\intertext{Similar to \eqref{eq:rep-formula-nu-infty}, by the mean value theorem, \eqref{eq:p-gamma-L2}, and \eqref{eq:p-gamma-L-infinity}, we have}
		&\apprle\int_0^\infty\min\{ R^{-3}\norm{\action{p}\gamma}{L^2_{q,p}}^2(1-s),\norm{\action{p}^3\gamma}{L^\infty_{q,p}}^2\}\, \myd{R}\apprle (1-s)^{1/4}\longrightarrow 0
	\end{align*}
as $s\rightarrow 1-$. This completes the proof of Theorem \ref{thm:A}.
\end{proof}

\begin{remark}\label{rem:linear-scattering}
If $d\geq 3$, then we can prove linear scattering of the solution. Recall that $\mu$ is a solution to \eqref{eq:VP-rep-potential-reformulate} if and only if $\gamma$ satisfies
\[
\partial_s \gamma+p\cdot\nabla_q\gamma+(1-s^2)^{(d-4)/2}\lambda E[\gamma]\cdot\nabla_p\gamma=0,\quad s\in (-1,1),
\]

If we define 
\[ \Phi(s,q,p):=(q+(s-1)p,p),\quad \nu(s,q,p)=\gamma(s,\Phi(s,q,p)),\]
then we have 
\[ \partial_s \nu =\partial_s \gamma+p\cdot \nabla_q\gamma=-\lambda (1-s^2)^{\frac{d-4}{2}} E\cdot\nabla_p \gamma,\]
which implies
\[ \norm{\partial_s \nu}{\Leb{\infty}_{q,p}}\apprle \varepsilon_0^3 (1-s^2)^{\frac{d-4}{2}}.\]
Note that $(1-s^2)^{\frac{d-4}{2}}$ is integrable on $[0,1)$ if $d\geq 3$. Hence the solution scatters linearly.
\end{remark}

\section{Construction of the wave operator}\label{sec:wave}

In this section, we will prove Theorems \ref{thm:B} and \ref{thm:c}, the construction of wave operators and scattering map. To construct the wave operator, we need to solve the Cauchy problem of \eqref{eq:VP-rep-potential-reformulate} at infinity or equivalently the Cauchy problem of \eqref{eq:VP-lens} starting from $s=1$. Since the Vlasov--Poisson system is time reversible, we consider the Cauchy problem of \eqref{eq:VP-lens} from $s=-1$. To mitigate the singularity,  we define the type-3 generating function
\begin{equation*}
	\begin{split}
S(s,w,p)=w\cdot p+\frac{|p|^2}{2}(s+1)-\lambda F(s)\phi_{-1}(w),		
	\end{split}
\end{equation*}
which produces the following symplectic change of variables:
\begin{equation}\label{eq:change-of-variable-sigma}
	\begin{aligned}
		q &= w+(s+1)z+\lambda F(s)(s+1)E_{-1}(w), &\quad p &= z+\lambda F(s)E_{-1}(w),\\
		w &= q-(s+1)p, &\quad z &= p-\lambda F(s)E_{-1}(q-(s+1)p),
	\end{aligned}
\end{equation}
where the new Hamiltonian is $\mathcal{K}=\H-\pr_s S=-\lambda f(s)[\phi(s,q)-\phi_{-1}(w)]$ and the Jacobian is
\begin{align*} 
	\frac{\pr(w,z)}{\pr(q,p)}=\begin{pmatrix}
		I_2 & -(s+1)I_2 \\
		-\lambda F(s)\nabla E_{-1}(q-(s+1)p) & I_2+\lambda F(s)(s+1)\nabla E_{-1}(q-(s+1)p) 
	\end{pmatrix},
\end{align*}
where $E_{-1}=\nabla\phi_{-1}$ and the  vector fields are
\begin{equation}\label{eq:expression-dK}
	\begin{aligned}
	&\nabla_w\mathcal{K}=-\lambda f(s)[E(s,q)-E_{-1}(w)]-\lambda^2 (s+1)f(s)F(s)\nabla E_{-1} (w)E(s,q), \\ &\nabla_z\mathcal{K}=-\frac{\lambda }{1-s} E(s,q) .
\end{aligned}
\end{equation}
Note that this change of variables is similar to \eqref{eq:nu} and preserves volume. If we define
\[ \sigma(s,w,z):=\gamma (s,q,p), \] 
then $\sigma$ solves
\begin{equation}\label{eq:sigma}
	\pr_s\sigma+\{\sigma,\KK\}=0,
\end{equation}
where $\{\sigma,\KK\}=\nabla_w\sigma\cdot \nabla_z\KK-\nabla_z\sigma\cdot\nabla_w\KK$.
To state our theorem, we define the electric field	
\begin{equation*}
	\begin{aligned}
E[\sigma](s, Q) & =E(s, Q)=\frac{1}{2\pi}\iint_{\mathbb{R}^2_w\times\mathbb{R}^2_z} \frac{Q-q(s, w, z)}{|Q-q(s, w, z)|^2} \sigma^2(s, w, z) \,\mathrm{d}w \,\mathrm{d} z \\
& =\frac{1}{2\pi}\iint_{\mathbb{R}^2_q\times\mathbb{R}^2_p} \frac{Q-q}{|Q-q|^2} \gamma^2(s, q, p) \,\mathrm{d}q \,\mathrm{d}p.
\end{aligned}
\end{equation*}
 Now we are ready to state the main theorem of this section.
\begin{theorem}\label{thm:wave-operator}
Assume the ``initial'' data $\sigma_{-1}$ and $E_{-1}:=E[\sigma_{-1}]$ satisfy
\begin{equation}\label{eq:bound-E0}
	\begin{split}
	\norm{E_{-1}}{W^{3,\infty}}	\leq c_0^2,
	\end{split}
\end{equation}
and 
\begin{equation}\label{eq:bound-sigma0}
\left\|\sigma_{-1}\right\|_{L_{w, z}^2}+\left\|\langle z\rangle^4 \sigma_{-1}\right\|_{L_{w, z}^{\infty}}+\sum_{0 \leq m+n \leq 2}\left\|\langle z\rangle^m \nabla_z^m \nabla_w^n \sigma_{-1}\right\|_{L_{w, z}^{\infty}} \leq c_0 .	
\end{equation}
Then there exists  a unique solution $\sigma \in C_s^0\left(\left[-1, T(c_0)\right]; L_{w, z}^2\right)$ of \eqref{eq:sigma}  with ``initial" data $\sigma(s=-1)=\sigma_{-1}$, and such that $(1+s) \partial_s \sigma, \nabla_{w, z} \sigma \in C_{s, w, z}^0$. Moreover, for $-1 \leq s\leq T(c_0)$ we have that for any $\ell \in \mathbb{N}$,
\begin{align}
&\norm{\sigma(s)}{\Leb{2}_{w,z}}+\norm{\action{z}^4 \sigma(s)}{\Leb{\infty}_{w,z}}+\norm{\nabla_{w,z}\sigma(s)}{\Leb{\infty}_{w,z}}\apprle c_0,\label{eq:sigma-c0} \\
&{\norm{\action{w,z}^{\ell} \sigma(s)}{\Leb{r}_{w,z}}\apprle \norm{\action{w,z}^l \sigma_{-1}}{\Leb{r}_{w,z}},\quad r\in \{2,\infty\}.}\label{eq:sigma-sigma0}
\end{align}
In particular, if $c_0$ is small enough, we can take $T(c_0)=0$.
\end{theorem}

The solution will be constructed by  Picard iteration. To do this, we need to obtain a priori estimates of the particle distribution functions and the vector fields. We make the following bootstrap assumptions on $\sigma$: given $c_0$ as in Theorem \ref{thm:wave-operator}, we assume that for $-1\leq s\leq T(c_0)$, there holds that
\begin{equation}\label{eq:bootstrap-assumptions-sigma}
\begin{aligned}
&\|\sigma(s)\|_{L_{w, z}^2}+\left\|\langle z\rangle^4 \sigma(s)\right\|_{L_{w, z}^{\infty}} \leq A_1 \leq 4 c_0, \\
&\left\|\nabla_w \sigma(s)\right\|_{ L_{w, z}^{\infty}}+\left\|\theta \nabla_z \sigma(s)\right\|_{L_{w, z}^{\infty}} \leq A_2 \leq 4 c_0, \\
&\left\|\nabla_{w, w}^2 \sigma(s)\right\|_{L_{w, z}^{\infty}}+\left\|\theta \nabla_{w, z}^2 \sigma(s)\right\|_{L_{w, z}^{\infty}}+\left\|\theta^2 \nabla_{z, z}^2 \sigma(s)\right\|_{L_{w, z}^{\infty}} \leq A_3 \leq 4 c_0,
\end{aligned}
\end{equation}
where $\theta$ is defined by
\begin{equation*}
	\begin{split}
	\theta(s,z):= \frac{\langle z\rangle}{1+(s+1)\langle z\rangle},	\quad  \frac{1}{2}\min\{\langle z\rangle,(s+1)^{-1}\}\leq \theta(s,z)\leq \min\{ \langle z\rangle,(1+s)^{-1} \},
	\end{split}
\end{equation*}
which satisfies the following differential equations:
$$
\partial_s \theta=-\theta^2 \quad \text { and } \quad \nabla_z \theta=\left(\frac{z}{\langle z\rangle^3}\right) \cdot \theta^2 .
$$
The organization of this section is as follows. In Section \ref{subs:comm-relationship}, we compute the commutative relations that will be needed to propagate the regularity. In Section \ref{subs:bootstrap}, we prove the a priori estimates \eqref{eq:bootstrap-assumptions-sigma} by a bootstrap argument. In Section \ref{subs:contruct-local-solutions}, we construct the local solution by Picard iteration.

\subsection{Commutative relations}\label{subs:comm-relationship}

Writing $\mathfrak{L}=\partial_s +\{\cdot,\mathcal{K}\}$, for the moments in $w$, $z$, we have the commutative relations 
\begin{align}\label{eq:evolution-of-weighted-sigma}
\mathfrak{L}[\langle z\rangle^k\sigma]=-k\sigma \langle z\rangle^{k-2}z\cdot \nabla_w\KK,\quad \mathfrak{L}[\langle w\rangle^k\sigma]=k\sigma\langle w\rangle^{k-2}w\cdot\nabla_z\KK.
\end{align}
For the derivatives, we have
\[
\mathfrak{L}\binom{\nabla_w \sigma}{\nabla_z \sigma}=\left(\begin{array}{cc}
-\nabla_w \nabla_z \mathcal{K} & \nabla_w^2 \mathcal{K} \\
-\nabla_z^2 \mathcal{K} & \nabla_w \nabla_z \mathcal{K}
\end{array}\right)\binom{\nabla_w \sigma}{\nabla_z \sigma}
\]
with
\begin{equation}\label{eq:expression-d^2K}
\begin{aligned}
\nabla_{w^j w^k}^2 \mathcal{K}= & -\lambda f(s) \partial_j\left[E_k(s, q)-E_{-1, k}(w)\right] \\
& -\lambda^2 F(s)(1-s)^{-1}\left\{\partial_j E(s, q) \cdot \partial_k E_{-1}(w)\right. \\
& \left.+\partial_k E(s, q) \cdot \partial_j E_{-1}(w)+\partial_j \partial_k E_{-1}(w) \cdot E(s, q)\right\} \\
& -\lambda^3 \frac{1+s}{1-s} F(s)^2 \partial_k E_{-1, a}(w) \partial_j E_{-1, b}(w) \cdot \partial_a E_b(s, q), \\
\nabla_w \nabla_z \mathcal{K}= & -\frac{\lambda}{1-s} \nabla E(s, q)-\lambda^2 \frac{1+s}{1-s} F(s)(\nabla E(s, q) \cdot \nabla) E_{-1}(w), \\
\nabla_z^2 \mathcal{K}=- & \lambda \frac{1+s}{1-s} \nabla E(s, q) .
\end{aligned}
\end{equation}
Observe that the first term in $\nabla_{w,w}^2\KK$ is not uniformly  integrable in $L^{\infty}_{z,w}$ near $s=-1$ due to \eqref{eq:pz-relation}. This effect will be controlled by introducing the weight $\theta$.   Lastly, we compute the commutative relations of the second order derivatives of $\sigma$.

\begin{equation}\label{eq:commutative-relation-d^2sigma}
\begin{split}
&\relphantom{=}\mathfrak{L}(\partial_{w^j}\partial_{w^k}\sigma)\\
&=\theta^{-1}\nabla_w\partial_{w^j}\mathcal{K}\cdot (\theta\nabla_z\partial_{w^k}\sigma)+\theta^{-1}\nabla_w\partial_{w^k}\mathcal{K}\cdot(\theta\nabla_z\partial_{w^j}\sigma)-\nabla_z\partial_{w^j}\mathcal{K}\cdot\nabla_w\partial_{w^k}\sigma\\
&\quad-\nabla_z\partial_{w^k}\mathcal{K}\cdot\nabla_w\partial_{w^j}\sigma+\theta^{-1}\nabla_w\partial_{w^j}\partial_{w^k}\mathcal{K}\cdot(\theta\nabla_z\sigma)-\nabla_z\partial_{w^j}\partial_{w^k}\mathcal{K}\cdot\nabla_w\sigma,\\
&\relphantom{=}\mathfrak{L}(\theta\partial_{z^j}\partial_{w^k}\sigma)\\
&=\mathfrak{L}(\ln\theta)\cdot\theta\partial_{z^j}\partial_{w^k}\sigma+\theta^{-1}\nabla_w\partial_{w^k}\mathcal{K}\cdot(\theta^2\nabla_z\partial_{z^j}\sigma)-\nabla_z\partial_{w^k}\mathcal{K}\cdot(\theta\nabla_w\partial_{z^j}\sigma)\\
&\quad+\nabla_w\partial_{z^j}\mathcal{K}\cdot (\theta\nabla_z\partial_{w^k}\sigma)-(\theta\nabla_z\partial_{z^j}\mathcal{K})\cdot\nabla_w\partial_{w^k}\sigma\\
&\quad+\nabla_w\partial_{z^j}\partial_{w^k}\mathcal{K}\cdot(\theta\nabla_z\sigma)-\theta\nabla_z\partial_{z^j}\partial_{w^k}\mathcal{K}\cdot\nabla_w\sigma,\\
&\relphantom{=}\mathfrak{L}(\theta^2\partial_{z^j}\partial_{z^k}\sigma)\\
&=2\mathfrak{L}(\ln\theta)\cdot\theta^2\partial_{z^j}\partial_{z^k}\sigma+\nabla_w\partial_{z^k}\mathcal{K}\cdot(\theta^2\nabla_z\partial_{z^j}\sigma)-\theta\nabla_z\partial_{z^k}\mathcal{K}\cdot(\theta\nabla_w\partial_{z^j}\sigma)\\
&\quad+\nabla_w\partial_{z^j}\mathcal{K}\cdot (\theta^2\nabla_z\partial_{z^k}\sigma)-(\theta\nabla_z\partial_{z^j}\mathcal{K})\cdot (\theta\nabla_w\partial_{z^k}\sigma)\\
&\quad+\theta\nabla_w\partial_{z^j}\partial_{z^k}\mathcal{K}\cdot(\theta\nabla_z\sigma)-\theta^2\nabla_z\partial_{z^j}\partial_{z^k}\mathcal{K}\cdot\nabla_w\sigma,
\end{split}
\end{equation}

\subsection{Closing the bootstrap}\label{subs:bootstrap} We first provide an outline of how to close the bootstrap.  To avoid the singularity at $s=1$, we will always assume $T\leq 0$. Lemma \ref{lemma:bound-K} shows that the bootstrap assumption \eqref{eq:bootstrap-assumptions-sigma} implies bounds on the vector fields:
\begin{equation}\label{eq:estimate-dK}
\begin{split}
	|\nabla_z\KK(w,z)|&\apprle 2c_0^2,\\
	|\nabla_w\KK(w,z)|&\apprle c_0^2\{\min\{(1+s)^{-1},|z|\}+ \langle F(s)\rangle^3\},
\end{split}
\end{equation}
and for the second order derivative
\begin{equation}\label{eq:estimate-d^2K}
\left\|\nabla_w \nabla_z \mathcal{K}\right\|_{L_{w, z}^{\infty}}+\left\|\theta \nabla_z \nabla_z \mathcal{K}\right\|_{L_{w, z}^{\infty}}+\left\|\theta^{-1} \nabla_w \nabla_w \mathcal{K}\right\|_{L_{w, z}^{\infty}} \leq c_0^2 \action{F(s)}^4 ,
\end{equation}
and for the third order derivative
\begin{equation}\label{eq:estimate-d^3K}
\begin{split}
&\Vert \theta^{-1}\nabla^3_{w,w,w}\mathcal{K}\Vert_{L^\infty_{w,z}}+\Vert \nabla^3_{w,w,z}\mathcal{K}\Vert_{L^\infty_{w,z}}+\Vert \theta\nabla^3_{w,z,z}\mathcal{K}\Vert_{L^\infty_{w,z}}+\Vert \theta^2\nabla^3_{z,z,z}\mathcal{K}\Vert_{L^\infty_{w,z}}\\
&\le c_0^2\langle F(s)\rangle^5.
\end{split}
\end{equation}
Then, Lemma \ref{lem:particle-density-bootstrap} shows that these estimates help to improve the bootstrap assumption \eqref{eq:bootstrap-assumptions-sigma}. Now we begin with the proof of the controls on $\mathcal{K}$ and its derivatives.
\begin{lemma}\label{lemma:bound-K}
	Let $\sigma\in C^0([-1,T],L^2_{w,z})$ with $\sigma(-1)=\sigma_{-1}$ such that 
\begin{equation}\label{eq:continuity-gamma}
	\begin{split}
\pr_s\rho+\div_q {\bf{j}}=0 \quad \mathrm{in} \, [-1,T]\times\R^2,		
	\end{split}
\end{equation}	
where
$$
\rho(s, q)=\int_{\mathbb{R}^2} \gamma^2(s, q, p) \myd{p} \quad \text { and } \quad \mathbf{j}(s, q)=\int_{\mathbb{R}^2} p \gamma^2(s, q, p) \myd{p}
$$
Suppose that $E_{-1}$ satisfies \eqref{eq:bound-E0}. Then there exists $T^*\left(c_0\right) \in(-1, T]$ such that
\begin{enumerate}[\rm (i)\,]
	\item Assume the first line of \eqref{eq:bootstrap-assumptions-sigma} holds. Then \eqref{eq:estimate-dK} holds for $-1<s<T^*$.
	\item Assume the last two lines of \eqref{eq:bootstrap-assumptions-sigma} hold. Then \eqref{eq:estimate-d^2K} holds for $-1<s<T^*$.
	\item Assume all the three lines of \eqref{eq:bootstrap-assumptions-sigma} hold. Then \eqref{eq:estimate-d^3K} holds for $-1<s<T^*$.
\end{enumerate}
\end{lemma}

The following lemma will help us to prove the bootstrap improvement.
\begin{lemma}\label{lemma:Lip-E}
	Under the assumptions of Lemma \ref{lemma:bound-K}, we have the following estimates on $E=E[\gamma]$ and its derivatives for fixed $-1<s_0<s_1<T^*$:
\begin{enumerate}[\rm (i)\,]
	\item 
\begin{equation*}
\begin{aligned}
\norm{E(s_1)-E(s_0)}{L^\infty_q} &\apprle \left\langle\ln \left(s_1-s_0\right)\right\rangle\left(s_1-s_0\right)\|\mathbf{j}\|_{L_{s, q}^{\infty}} \\
& +\left(s_1-s_0\right)^2\left[\left\|\langle p\rangle^2 \gamma\right\|_{L_{s, q, p}^{\infty}}^2+\|\gamma\|_{L_s^{\infty} L_{q, p}^2}^2\right] .
\end{aligned}
\end{equation*}
	\item
\begin{equation*}
\begin{aligned}
\left\|\nabla_q E\left(s_1\right)-\nabla_q E\left(s_0\right)\right\|_{L_q^{\infty}} \apprle & \left\langle\ln \left(s_1-s_0\right)\right\rangle\left(s_1-s_0\right)\left\|\nabla_q \mathbf{j}\right\|_{L_{s, q}^{\infty}} \\
& +\left(s_1-s_0\right)^2\left[\left\|\langle p\rangle^3 \gamma\right\|_{L_{s, q, p}^{\infty}}^2\right. \\
& \left.+\left\|\nabla_q \gamma\right\|_{L_{s, q, p}^{\infty}}^2+\|\gamma\|_{L_s^{\infty} L_{q, p}^2}^2\right].
\end{aligned}
\end{equation*}

\item 

\begin{equation*}
\begin{aligned}
&\relphantom{=}\left\|\nabla_q^2 E\left(s_1\right)-\nabla_q^2 E\left(s_0\right)\right\|_{L_q^{\infty}}  \\
&\apprle \left\langle\ln \left(s_1-s_0\right)\right\rangle\left(s_1-s_0\right)\left\|\nabla_q^2 \mathbf{j}\right\|_{L_{s, q}^{\infty}} \\
& +\left(s_1-s_0\right)^2\left[\|\gamma\|_{L_s^{\infty} L_{q, p}^2}^2+\left\|\langle p\rangle^3 \gamma\right\|_{L_{s, q, p}^{\infty}}\left\|\nabla_q^2 \gamma\right\|_{L_{s, q, p}^{\infty}}+\left\|\langle p\rangle^{2} \nabla_q \gamma\right\|_{L_{s, q, p}^{\infty}}^2\right] .
\end{aligned}
\end{equation*}

\end{enumerate}
\end{lemma}

\begin{proof}
Let us only prove (ii), since the proofs for (i) and (iii) are similar. Integrating \eqref{eq:continuity-gamma} over $[s_1,s_2]\times \R^2$ against the test function $R^{-1}\nabla^2\chi(R^{-1}(x-y))$, we have
\begin{equation*}
	\begin{aligned}
&\iint_{\mathbb{R}^2_q\times\mathbb{R}^2_p} R^{-1}[\gamma^2(s_2,q,p)-\gamma^2(s_1,q,p)](\nabla^2\chi)(R^{-1}(x-q))\,\mathrm{d}q\,\mathrm{d}p\\
&+\int_{s_1}^{s_2}\int_{\mathbb{R}^2_q}(\nabla^3\chi)(R^{-1}(x-q))R^{-2} {\bf j}(q)\,\mathrm{d}q\,\mathrm{d}s=0		.
	\end{aligned}
\end{equation*}	
Using integration by parts, we have
\begin{equation*}
	\begin{split}
\nabla E_{R}(s_2)-\nabla E_{R}(s_1)=\int_{s_1}^{s_2}\int_{\mathbb{R}^2} R^{-1}(\nabla^2\chi)(R^{-1}(x-q))\nabla{\bf j}\,\mathrm{d}q\,\mathrm{d}s.		
	\end{split}
\end{equation*}
Hence,
\begin{equation*}
	\begin{split}
|\nabla E_{R}(s_2)-\nabla E_{R}(s_1)|\apprle R\norm{\nabla_q\bf{j}}	{L^\infty_{s,q}}(s_2-s_1)	.
	\end{split}
\end{equation*}
Using
\begin{equation*}
	\begin{split}
	|\nabla E_{R}|\apprle	\min\{{R^2\norm{\nabla_q\gamma}{L^\infty_{s,q,p}}\norm{\langle p\rangle^3\gamma}{L^\infty_{s,q,p}}, R^{-1}\norm{\gamma}{L^\infty_sL^2_{q,p}}^2  }\},
	\end{split}
\end{equation*}
we have
\begin{align*}
|\nabla E(s_2)-\nabla E(s_1)|&\leq\int_{0}^A	\norm{\nabla_q\gamma}{L^\infty_{s,q,p}}\norm{\langle p\rangle^3\gamma}{L^\infty_{s,q,p}}\mathrm{d}R+\int_{A^{-1/2}}^\infty\norm{\gamma}{L^\infty_sL^2_{q,p}}^2\frac{\mathrm{d}R}{R^3}\\
&\relphantom{=}+\int_{A}^{A^{-1/2}}(s_2-s_1)\norm{\nabla_q\bf{j}}{L^\infty_{s,q}} \frac{\mathrm{d}R}{R}.
\end{align*}
The proof follows by taking $A=(s_2-s_1)^2$.
\end{proof}

Now, we are ready to prove Lemma \ref{lemma:bound-K}.

\begin{proof}[Proof of Lemma \ref{lemma:bound-K}]
(i) Recall from \eqref{eq:change-of-variable-sigma} that
\begin{equation}\label{eq:pz-relation}
|p-z|\apprle	c_0^2	|F(s)|, \quad |q-w|\leq (s+1)(|z|+|F(s)|c_0^2)
\end{equation}
hold and the change of variable preserves volume. These imply that for any $\beta \geq 0$, we have
\begin{equation}\label{eq:bound-weight-sigma}
\begin{aligned}
& \|\gamma(s)\|_{L_{q, p}^r}=\|\sigma(s)\|_{L_{q, p}^r}, \\
& \left\||p|^\beta \gamma(s)\right\|_{L_{q, p}^r} \apprle c_0^{2 \beta}  |F(s)|^\beta\|\sigma(s)\|_{L_{w, z}^r}+\left\||z|^\beta \sigma(s)\right\|_{L_{w, z}^r}
\end{aligned}
\end{equation}
In addition, since $\frac{\partial z}{\partial p}=\mathrm{Id}+O\left(c_0^2(1+s)F(s)  \right)$ has bounded Jacobian, we see that
\begin{equation}\label{eq:bound-j}
	\|\mathbf{j}(s)\|_{L_q^{\infty}} \leq\left\|\int\left[|z|+c_0^2 |F(s)|\right] \sigma^2 d z\right\|_{L_w^{\infty}} \apprle\action{c_0}^2 \action{F(s)} \Vert\langle z\rangle^2 \sigma \|_{L_{w, z}^{\infty}}^2
\end{equation}
and using Lemma \ref{lemma:Lip-E} (i), \eqref{eq:bound-weight-sigma}, and \eqref{eq:bound-j}, we obtain that for $2^{-k-1}-1 \leq s_2 \leq s_1 \leq 2^{-k}-1$,
$$
\left\|E\left(s_2, q\right)-E\left(s_1, q\right)\right\|_{L_q^{\infty}} \apprle\left\langle c_0\right\rangle^2 2^{-k}\action{k}^2\left(\left\|\action{z}^2 \sigma\right\|_{L_{s, w, z}^{\infty}}^2+\|\sigma\|_{L_s^{\infty} L_{w, z}^2}^2\right).
$$
Therefore, $E\left(2^{-k}-1\right)$ is Cauchy in $L_q^{\infty}$ and hence it follows that
\begin{equation}\label{eq:E-E1-difference}
\left\|E(s, q)-E_{-1}(q)\right\|_{L_q^{\infty}} \apprle\left\langle c_0\right\rangle^2 (s+1) \action{F(s)}^2 \left(\left\| \action{z}^2 \sigma\right\|_{L_{s, w, z}^{\infty}}^2+\|\sigma\|_{L_s^{\infty} L_{w, z}^2}^2\right),
\end{equation}
so that by \eqref{eq:bootstrap-assumptions-sigma}, we have
$$
|E(s, q)| \apprle\left|E(s, q)-E_{-1}(q)\right|+\left|E_{-1}(q)\right| \apprle c_0^2,
$$
which gives the desired bound for $\nabla_z\KK$ by \eqref{eq:expression-dK}. 

To bound $\nabla_w\KK$, it follows from \eqref{eq:bound-E0}, \eqref{eq:bootstrap-assumptions-sigma}, \eqref{eq:E-E1-difference}, and the mean value theorem with \eqref{eq:pz-relation} that
\begin{equation}\label{eq:E-E-1-difference}
	\begin{split}
		|E(s,q)-E_{-1}(w)|&\leq |E(s,q)-E_{-1}(q)|+|E_{-1}(w)-E_{-1}(q)|\\
		&\apprle c_0^2(s+1)\langle F(s)\rangle^2+c_0^2\min\{1, (s+1)(|z|+|F(s)|c_0^2)  \}.
	\end{split}
\end{equation}
Hence by \eqref{eq:expression-dK}, we obtain the desired bound for $\nabla_w\mathcal{K}$ as well. 

(ii) By the change of variables \eqref{eq:change-of-variable-sigma}, we have
\begin{equation*}
	\begin{split}
		\nabla_q\gamma=\nabla_w\sigma-\nabla_z\sigma\cdot\lambda F(s)\nabla E_{-1}(q-(s+1)p)
	\end{split}
\end{equation*}
so that 
\begin{equation*}
	\begin{split}
	\norm{\nabla_q\gamma}{L^r_{q,p}}\apprle \action{c_0}^2\langle F(s) \rangle	\norm{\nabla_{w,z}\sigma}{L^r_{w,z}}
	\end{split}
\end{equation*}
for $r\in \{2,\infty\}$ and
$$
\begin{aligned}
\left\|\nabla_q \mathbf{j}(s)\right\|_{L_q^{\infty}} & \apprle \action{c_0}^2\langle F(s)\rangle\left\|\int\left[|z|+c_0^2\langle F(s)\rangle\right]|\sigma| \cdot\left|\nabla_{w, z} \sigma\right| \,\mathrm{d} z\right\|_{L_w^{\infty}} \\
& \apprle \action{c_0}^4\langle F(s)\rangle^2\left[\left\|\langle z\rangle^4 \sigma\right\|_{L_{w, z}^{\infty}}^2+\left\|\nabla_{w, z} \sigma\right\|_{L_{w, z}^{\infty}}^2\right].
\end{aligned}
$$
For $2^{-k-1}-1 \leq s_2 \leq s_1 \leq 2^{-k}-1$, Lemma \ref{lemma:Lip-E} (ii)  and \eqref{eq:bound-weight-sigma} give 
$$
\begin{aligned}
\left\|\nabla_q E\left(s_2, q\right)-\nabla_q E\left(s_1, q\right)\right\|_{L_q^{\infty}} \apprle & \left\langle c_0\right\rangle^4 \action{k}^3 2^{-k} \cdot\left(\left\|\langle z\rangle^4 \sigma\right\|_{L_{s, w, z}^{\infty}}^2+\left\|\nabla_{w, z} \sigma\right\|_{L_{s, w, z}^{\infty}}^2\right) \\
& +c_0^{10} 2^{-\frac{3 k}{2}}\left[\left\|\langle z\rangle^4 \sigma\right\|_{L_{s, w, z}^{\infty}}^2+\|\sigma\|_{L_s^{\infty} L_{w, z}^2}^2\right].
\end{aligned}
$$
Applying similar arguments as in (i) and \eqref{eq:bootstrap-assumptions-sigma}, we obtain
\begin{equation}\label{eq:gradient-control-s-wave}
\begin{aligned}
&\relphantom{=}\left\|\nabla_q E(s, q)-\nabla_q E_{-1}(q)\right\|_{L_q^{\infty}} \\
 &\apprle \left\langle c_0\right\rangle^4 (1+s)\langle F(s)\rangle^3 \cdot\left(\left\|\langle z\rangle^4 \sigma\right\|_{L_{s, w, z}^{\infty}}^2+\left\|\nabla_{w, z} \sigma\right\|_{L_{s, w, z}^{\infty}}^2\right) \\
&\relphantom{=} +c_0^{10} (1+s)^{\frac{3}{2}}\left[\left\|\langle z\rangle^4 \sigma\right\|_{L_{s, w, z}^{\infty}}^2+\|\sigma\|_{L_s^{\infty} L_{w, z}^2}^2\right] \\
 &\apprle c_0^2 (1+s)\langle F(s)\rangle^3 .
\end{aligned}
\end{equation}
Using the formulas in \eqref{eq:bootstrap-assumptions-sigma}, \eqref{eq:expression-d^2K}, \eqref{eq:pz-relation}, and \eqref{eq:gradient-control-s-wave}, we see that
\begin{equation*}
\begin{aligned}
 \theta\left|\nabla_{z, z}^2 \mathcal{K}\right| &\leq|\nabla E| \cdot (1+s) \min \left\{(1+s)^{-1},\action{z}\right\} \apprle c_0^2 \\
\left|\nabla_{w, z}^2 \mathcal{K}\right| &\leq|\nabla E|\left(1+(s+1)\langle F(s)\rangle\left|\nabla E_{-1}\right|\right) \apprle c_0^2 ,
\end{aligned}
\end{equation*}
and
\begin{equation*}
\begin{aligned}
\left|\nabla_{w, w}^2 \mathcal{K}\right|  &\leq (1+s)^{-1}\left|\nabla E_{-1}(q)-\nabla E_{-1}(w)\right|+(1+s)^{-1}\left|\nabla E(s, q)-\nabla E_{-1}(q)\right| \\
&\relphantom{=} +\langle F (s)\rangle\left[2\left|\nabla E_{-1}\right|^2|\nabla E|+\left|\nabla^2 E_{-1}\right||E|\right]+(1+s)\langle  F(s)\rangle^2\left|\nabla E_{-1}\right|^2|\nabla E| \\
 &\apprle c_0^2 \min \left\{(1+s)^{-1},|z|\right\}+c_0^2\langle F(s)\rangle^4,
\end{aligned}
\end{equation*}
which imply \eqref{eq:estimate-d^2K}.

(iii) The proof is similar to (ii). Starting from
 \begin{equation*}
\begin{gathered}
\nabla_z\mathcal{K}=-\frac{\lambda}{1-s} E(q),\qquad \frac{\partial q^k}{\partial z^j}=(s+1)\delta_j^k,\\
\frac{\partial q^k}{\partial w^j}=\delta_j^k+\lambda (s+1)F(s)\partial_j\partial_k\phi_{-1}(w),
\end{gathered}
\end{equation*}
we deduce
\begin{align*}
\theta^2\vert \nabla^3_{z,z,z}\mathcal{K}\vert&\leq ((s+1)\theta)^2\vert\nabla^2E(q)\vert,\\
\theta\vert\nabla^3_{w,z,z}\mathcal{K}\vert&\le ((s+1)\theta)\cdot\left[1+(s+1)\langle F(s)\rangle\vert\nabla E_{-1}\vert\right]\cdot \vert\nabla^2E(q)\vert,\\
\vert\nabla^3_{w,w,z}\mathcal{K}\vert&\le\left[1+(s+1)\langle F(s)\rangle\vert\nabla E_{-1}\vert\right]^2\cdot \vert\nabla^2E(q)\vert\\
&\relphantom{=}+\left[1+(1+s)\langle F(s)\rangle\vert\nabla E_{-1}\vert\right]\cdot\left[1+(1+s)\langle F(s)\rangle\vert\nabla^2 E_{-1}\vert\right]\cdot \vert\nabla E(q)\vert,
\end{align*}
and finally, from \eqref{eq:expression-d^2K}, we obtain that
\begin{equation*}
\begin{split}
\theta^{-1}\vert\nabla^3_{w,w,w}\mathcal{K}\vert&\le \left[(1+s)^{-1}+\langle F(s)\rangle\cdot \vert\nabla E_{-1}\vert\right]\cdot \vert\nabla^2E(s,q)-\nabla^2E_{-1}(w)\vert\\
&\quad+\langle F(s)\rangle\cdot \left[\vert\nabla^2E\vert\cdot\vert\nabla E_{-1}\vert+\vert\nabla E\vert\cdot \vert\nabla^2E_{-1}\vert+\vert\nabla^3E_{-1}\vert\cdot\vert E\vert\right]\\
&\quad+(1+s)\langle F(s)\rangle^2\cdot\left[\vert\nabla^2 E\vert\cdot\vert\nabla E_{-1}\vert^2+\vert\nabla E\vert \cdot\vert\nabla E_{-1}\vert\cdot\vert\nabla^2E_{-1}\vert\right]\\
&\quad+(1+s)^2\langle F(s)\rangle^3\cdot\left[\vert \nabla^2 E\vert\cdot\vert\nabla E_{-1}\vert^3\right].
\end{split}
\end{equation*}

On the other hand, we note that
\begin{equation*}
\begin{split}
\Vert\nabla^2{\bf j}(s)\Vert_{L^\infty_q} &\apprle \action{c_0}^2\langle  F(s)\rangle^2\norm*{\int\left[\vert z\vert+c_0^2\langle F(s)\rangle\right]\cdot\left[\vert \sigma\vert \cdot\vert\nabla^2_{w,z}\sigma\vert+\vert\nabla_{w,z}\sigma\vert^2\right] dz}{L^\infty_w}\\
&\apprle \action{c_0}^4\langle F(s)\rangle^3\left[\Vert \langle z\rangle^4\sigma\Vert_{L^\infty_{w,z}}\Vert\nabla^2_{w,z}\sigma\Vert_{L^\infty_{w,z}}+\Vert\langle z\rangle^{1.6}\nabla_{w,z}\sigma\Vert_{L^\infty_{w,z}}^2\right].
\end{split}
\end{equation*}
By Proposition \ref{prop:interpolation}, we have
\begin{equation*}
	\begin{split}
		\Vert\langle z\rangle^{1.6}\nabla_{w,z}\sigma\Vert_{L^\infty_{w,z}}\apprle A_1+A_3.
	\end{split}
\end{equation*}
Hence by Lemma \ref{lemma:Lip-E} and the bootstrap assumptions \eqref{eq:bootstrap-assumptions-sigma}, we get
\begin{equation*}
\begin{split}
\Vert\nabla^2E(s,q)-\nabla^2E_{-1}(w)\Vert_{L^\infty_{w,z}}&\apprle  c_0^2\langle  F(s)\rangle^5+c_0^2\min\{(1+s)^{-1},\vert z\vert\},
\end{split}
\end{equation*}
which implies \eqref{eq:estimate-d^3K}. This completes the proof of Lemma \ref{lemma:bound-K}.
\end{proof}

Finally,  we can  close the bootstrap.
\begin{lemma}\label{lem:particle-density-bootstrap}
	Suppose that $\sigma\in C([-1,T];L^2_{w,z})$ satisfies \eqref{eq:sigma} for some Hamiltonian $\KK$ satisfying \eqref{eq:estimate-dK} and \eqref{eq:estimate-d^2K}. If in addition $E_{-1}$ and $\sigma_{-1}$ satisfy \eqref{eq:bound-E0} and \eqref{eq:bound-sigma0}, then there exists $T(c_0)>-1$ such that \eqref{eq:bootstrap-assumptions-sigma} hold for $A_1=A_2=A_3=2c_0.$	 
\end{lemma}
\begin{proof}
	 The proof is a simple application of Lemma \ref{lem:Liouville} and Gr\"{o}nwall's inequality. We first compute the commutation relations:
\begin{equation}\label{eq:commutation-relation-theta}
\begin{aligned}
\mathfrak{L}\left(\theta \nabla_z \sigma\right) & =(\mathfrak{L} \ln \theta) \cdot \theta \nabla_z \sigma+\theta\left\{\nabla_z \mathcal{K}, \sigma\right\} \\
& =(\mathfrak{L} \ln \theta) \cdot \theta \nabla_z \sigma+\theta \nabla_z \sigma \cdot \nabla_w \nabla_z \mathcal{K}-\nabla_w \sigma \cdot \theta \nabla_{z } \nabla_z \mathcal{K}, \\
\mathfrak{L}\left(\nabla_w \sigma\right) & =\left\{\nabla_w \mathcal{K}, \sigma\right\}=\theta \nabla_z \sigma \cdot \theta^{-1} \nabla_w \nabla_w \mathcal{K}-\nabla_w \sigma \cdot \nabla_z \nabla_w \mathcal{K} .
\end{aligned}
\end{equation}	 
Using \eqref{eq:evolution-of-weighted-sigma}, we find that
\begin{equation*}
\left\|\langle z\rangle^m \sigma(s)\right\|_{L_{w, z}^r} \leq\left\|\langle z\rangle^m \sigma_{-1}\right\|_{L_{w, z}^r}+m \int_{-1}^s\left\|\langle z\rangle^{-1} \nabla_w \mathcal{K}\left(s^{\prime}\right)\right\|_{L_{w, z}^{\infty}}\left\|\langle z\rangle^m \sigma\left(s^{\prime}\right)\right\|_{L_{w, z}^r} \mathrm{~d} s^{\prime}
\end{equation*} 
and we can easily propagate the first line of \eqref{eq:bootstrap-assumptions-sigma}. 

For the derivatives, we also need to control $\theta$. We note that
\begin{equation*}
	\begin{split}
	\mathfrak{L}(\ln \theta)=-\of{1+\frac{z}{\action{z}^3}\nabla_w\KK}\theta\apprle c_0^2+c_0^2\langle F(s)\rangle^3.	
	\end{split}
\end{equation*}
and we deduce from Lemma \ref{lem:Liouville}, \eqref{eq:estimate-d^2K}, and \eqref{eq:commutation-relation-theta} that
\begin{equation*}
\begin{aligned}
&\relphantom{=}\left\|\theta \nabla_z \sigma(s)\right\|_{L_{w, z}^r} \\
& \leq\left\|\theta \nabla_z \sigma_{-1}\right\|_{L_{w, z}^r}+c_0^2 \int_{-1}^s\left\langle F\left(s^{\prime}\right)\right\rangle^4\left\{\left\|\theta \nabla_z \sigma\left(s^{\prime}\right)\right\|_{L_{w, z}^r}+\left\|\nabla_w \sigma\left(s^{\prime}\right)\right\|_{L_{w, z}^r}\right\} \,\mathrm{d} s^{\prime}, \\
&\relphantom{=} \left\|\nabla_w \sigma(s)\right\|_{L_{w, z}^r}\\
& \leq\left\|\nabla_w \sigma_{-1}\right\|_{L_{w, z}^r}+c_0^2 \int_{-1}^s\left\langle F \left(s^{\prime}\right)\right\rangle^4\left\{\left\|\theta \nabla_z \sigma\left(s^{\prime}\right)\right\|_{L_{w, z}^r}+\left\|\nabla_w \sigma\left(s^{\prime}\right)\right\|_{L_{w, z}^r}\right\} \,\mathrm{d} s^{\prime},
\end{aligned}
\end{equation*} 
and this allows us to propagate the second line of \eqref{eq:bootstrap-assumptions-sigma} for a short time. Propagation of the last line of \eqref{eq:bootstrap-assumptions-sigma} follows similarly from \eqref{eq:estimate-d^3K} and \eqref{eq:commutative-relation-d^2sigma}.	 
\end{proof}

\subsection{Construction of local solutions}\label{subs:contruct-local-solutions}
This subsection is devoted to the proof of Theorem \ref{thm:wave-operator}. We construct the solution via Picard iteration which highly relies on the Hamiltonian structure of the system. 

\begin{proof}[Proof of Theorem \ref{thm:wave-operator}]
	Define $\sigma_{(0)}(s,w,z)=\sigma_{-1}(w,z)$. We first construct approximate solutions to the problem. Given $T\in (-1,1)$ and $\sigma_{(n)}\in C([-1,T];\Sob{1}{\infty}_{w,z})\cap C([-1,T];\Leb{2}_{w,z})$, consider the following system
\begin{equation}\label{eq:sigma-n-approx}
\begin{aligned}
&\partial_s \sigma_{(n+1)}+\{\sigma_{(n+1)},\mathcal{K}_n\}=0,\quad \sigma_{(n+1)}(-1)=\sigma_{-1}, \\
&\mathcal{K}_n:=-\lambda f(s) (\phi_n(s,q)-\phi_{-1}(w)),\\
& \phi_n(s,q):= \frac{1}{2\pi}\iint_{\mathbb{R}^2_p\times\mathbb{R}^2_y} \ln |q-y| \gamma_{(n)}^2(s,y,p)\myd{y}\myd{p},
\end{aligned}
\end{equation}
{{where $\gamma_{(n)}$}} is defined by $\sigma_{(n)}$ through \eqref{eq:change-of-variable-sigma}. Since $\mathcal{K}_n$ satisfies \eqref{eq:estimate-dK}, \eqref{eq:estimate-d^2K}, and \eqref{eq:estimate-d^3K}, it follows that we have the following bounds: since $\theta(s)\leq \action{z}$ and ${\theta}^{-1}(s)\leq 2$, we have (i) first order derivatives
\begin{equation}\label{eq:vector-field-conditions-growth}
\begin{aligned}
|\nabla_z\mathcal{K}_n(s,w,z)|&\apprle c_0^2\\
|\nabla_w\mathcal{K}_n(s,w,z)|&\apprle c_0^2\{\min\{(1+s)^{-1},|z|\}+\action{F(s)}^3\},
\end{aligned}
\end{equation}
(ii) second order derivatives
\begin{equation}\label{eq:vector-field-conditions}
\begin{aligned}
|\nabla_w\nabla_z\mathcal{K}_n(s,w,z)|+|\nabla_z\nabla_z\mathcal{K}_n(s,w,z)|&\apprle c_0^2 \action{F(s)}^4,\\
|\nabla_w\nabla_w\mathcal{K}_n(s,w,z)|&\apprle c_0^2\action{z}\action{F(s)}^4,
 \end{aligned}
\end{equation}
and (iii) third order derivatives:
\begin{equation}\label{eq:third-order-derivatives-approximate}
\begin{aligned}
|\nabla_{w,w,z}^3\mathcal{K}_n(s,w,z)|+|\nabla_{w,z,z}^3\mathcal{K}_n(s,w,z)|+|\nabla_{z,z,z}^3\mathcal{K}_n(s,w,z)|&\apprle c_0^2 \action{F(s)}^5,\\
|\nabla_{w,w,w}^3\mathcal{K}_n(s,w,z)|&\apprle c_0^2\action{z}^2\action{F(s)}^5.
\end{aligned}
\end{equation}

By \eqref{eq:vector-field-conditions-growth} and \eqref{eq:vector-field-conditions}, we can apply Carath\'eodory's extension of the Cauchy-Lipschitz theorem (see e.g. \cite[Theorem I.5.3]{H80}) to deduce that the following system
\begin{equation}\label{eq:characteristic-flow}
\dot{W}^n=\nabla_z \mathcal{K}_n;\quad W^n(s)=w\quad\text{and}\quad  \dot{Z}^n=-\nabla_w \mathcal{K}_n;\quad Z^n(s)=z
\end{equation}
admits a unique solution $(W^n,V^n)$ which is absolutely continuous in time. Furthermore, it follows from \eqref{eq:third-order-derivatives-approximate} that $(w,z)\mapsto (W^n(-1;s,w,z),Z^n(-1;s,w,z))$ is $C^{1,1}$ for fixed $s\in [-1,T]$.  Define 
\[ \sigma_{(n+1)}(s,w,z)=\sigma_{-1}(W^n(-1;s,w,z),Z^n(-1;s,w,z)).\]
Then $\sigma_{(n+1)}\in C([-1,T];\Sob{1}{\infty}_{w,z})$ satisfies \eqref{eq:sigma-n-approx} and $(1+s)\partial_s \sigma_{(n+1)}\rightarrow 0$ as $s\rightarrow -1$. Also, by Liouville's theorem, $\norm{\sigma_{(n+1)}(s)}{\Leb{2}_{w,z}}=\norm{\sigma_{-1}}{\Leb{2}_{w,z}}$. Since $\sigma_{(n+1)}$ satisfies \eqref{eq:sigma-n-approx}, one can show that $\sigma_{(n+1)}$ is weakly continuous in $\Leb{2}_{w,z}$ on $[-1,T]$ and hence it follows that  $\sigma_{(n+1)}\in C([-1,T];\Leb{2}_{w,z})$. 

By Lemmas \ref{lemma:bound-K} and \ref{lem:particle-density-bootstrap}, there exists $T(c_0)>-1$ such that \eqref{eq:bootstrap-assumptions-sigma} holds for all $s\in [-1,T(c_0)]$ with {$A_1=A_2=A_3=2c_0$.} Furthermore, since $\theta^{-1}\leq 2$, it follows that $\nabla_{w,z}\sigma_{(n+1)}$ and $\nabla^2_{w,z}\sigma_{(n+1)}\in \Leb{\infty}_{s,w,z}$. Moreover, one can show that \eqref{eq:sigma-c0} and \eqref{eq:sigma-sigma0} hold for $\sigma_{(n)}$ uniformly in $n$ by using the commutation relations given in Section \ref{subs:comm-relationship}. The above argument holds for $n=0$. The general $n$ follows by induction.

Next, we show that $\sigma_{(n)}$ forms a Cauchy sequence in $\Leb{\infty}_{s,w,z}$. Define  
\[ \delta_{(n)}=\sigma_{(n+1)}-\sigma_{(n)},\quad \delta \mathcal{K}_{(n)}=\mathcal{K}_n-\mathcal{K}_{n-1},\quad \mathfrak{L}_n :=\partial_s + \{\cdot,\mathcal{K}_n\},\quad \delta \mathfrak{L}_n=\{\cdot,\delta\mathcal{K}_{(n)}\} \]
so that 
\begin{equation}\label{eq:delta-n}
 \mathfrak{L}_n \delta_{(n)}=-\delta \mathfrak{L}_n \sigma_{(n)}.
\end{equation}

Note that 
\begin{equation}\label{eq:Kn-z-w}
\begin{aligned}
\nabla_z \delta \mathcal{K}_{(n)}&=-\lambda (1+s)f(s)(E_n(s,q)-E_{n-1}(s,q)),\\
\nabla_w \delta \mathcal{K}_{(n)}&=-\lambda f(s)(E_n(s,q)-E_{n-1}(s,q))\\
&\relphantom{=}-\lambda^2 F(s) (E_n(s,q)-E_{n-1}(s,q))\cdot \nabla E_{-1}(q).
\end{aligned}
\end{equation}

We claim that 
\begin{equation}\label{eq:delta-n-sequence}
\begin{split}
\Vert \nabla_{w,z}\delta\mathcal{K}_{(n)}(s)\Vert_{L^\infty_{w,z}}&\apprle  c_0\langle F(s)\rangle^6\norm{\delta_{(n-1)}(s)}{L^\infty_{w,z}\cap L^2_{w,z}},
\end{split}
\end{equation}
which will be proved soon.  Since $\delta_{(n)}$ satisfies \eqref{eq:delta-n}, using the Hamiltonian structure and the uniform estimate \eqref{eq:sigma-c0}, we get 
\begin{align*}
 \norm{\delta_{(n)}(s)}{\Leb{r}_{w,z}}&\apprle \int_{-1}^s \norm{\nabla_{w,z} \sigma_{(n)}}{\Leb{r}_{w,z}} \norm{\nabla_{w,z} \delta \mathcal{K}_{(n)}}{\Leb{\infty}_{w,z}} \myd{\tau}\\
 &\apprle c_0^2 \int_{-1}^s  \action{F(\tau)}^6   \left(\norm{\delta_{(n-1)}(\tau)}{\Leb{2}_{w,z}}+\norm{\delta_{(n-1)}(\tau)}{\Leb{\infty}_{w,z}}\right)\myd{\tau}
\end{align*}
for $r\in \{2,\infty\}$. If we choose $T$ sufficiently small, then one can easily show that $\sigma_{(n)}$ form a Cauchy sequence in $ \Leb{\infty}_s\Leb{2}_{w,z}\cap \Leb{\infty}_{s,w,z}$ (see the proof of Theorem \ref{thm:lwp}), and hence there exists $\sigma \in \Leb{\infty}_s\Leb{2}_{w,z}\cap \Leb{\infty}_{s,w,z}$ such that $\sigma_{(n)}\rightarrow \sigma$. Moreover, by the uniform convergence, we see that $\sigma \in C([-1,T];\Leb{2}_{w,z})$.

On the other hand, by interpolation (Proposition \ref{prop:interpolation} (i)), we have
\begin{align*}
\norm{\nabla_{w,z}\delta_{(n)}}{\Leb{\infty}_{w,z}}&\apprle \norm{\delta_{(n)}}{\Leb{\infty}_{w,z}}^{1/2}\left[\norm{\nabla_{w,z}^2 \sigma_{(n+1)}}{\Leb{\infty}_{w,z}}+\norm{\nabla_{w,z}^2 \sigma_{(n)}}{\Leb{\infty}_{w,z}}\right]^{1/2}.
\end{align*}
Since $\nabla^2_{w,z}\sigma_n$ is uniformly bounded in $\Leb{\infty}_{s,w,z}$, it follows that $\sigma_{(n)}$ forms a Cauchy sequence in $C_s \Sob{1}{\infty}_{w,z}$. Moreover, {{by Fatou's lemma or the conservation}, and \eqref{eq:evolution-of-weighted-sigma}}, one can show that $\sigma$ satisfies \eqref{eq:sigma-c0} and \eqref{eq:sigma-sigma0}. Since $\sigma_{(n)}$ converges to $\sigma$ in $C_s \Sob{1}{\infty}_{w,z}$ and $\sigma_{(n+1)}$ satisfies \eqref{eq:sigma-n-approx}, we see that $\sigma$ is a solution of \eqref{eq:sigma} with $\sigma(-1)=\sigma_{-1}$. {{The proof of uniqueness is straightforward and is therefore omitted}}.
Therefore, it remains to prove \eqref{eq:delta-n-sequence}.

First, it follows from \eqref{eq:change-of-variable-sigma} and \eqref{eq:bootstrap-assumptions-sigma} that 
\begin{align*}
&\relphantom{=}|E_n(s,q)-E_{n-1}(s,q)|\\
&\apprle \norm{\delta_{(n-1)}(s)}{\Leb{\infty}_{w,z}} \left(\norm{\action{p}^3\gamma_{(n)}}{\Leb{\infty}_{q,p}}+\norm{\action{p}^3\gamma_{(n-1)}}{\Leb{\infty}_{q,p}} \right)\\
&\relphantom{=}+\norm{\delta_{(n-1)}(s)}{\Leb{2}_{w,z}} \left(\norm{\gamma_{(n)}}{\Leb{2}_{q,p}}+\norm{\gamma_{(n-1)}}{\Leb{\infty}_{q,p}} \right)\\
&\apprle {{\action{c_0}^{6}}}\action{F(s)}^3 \norm{\delta_{(n-1)}(s)}{\Leb{\infty}_{w,z}} \left(\norm{\action{z}^3\sigma_{(n)}}{\Leb{\infty}_{w,z}}+\norm{\action{z}^3\sigma_{(n-1)}}{\Leb{\infty}_{w,z}} \right)\\
&\relphantom{=}+\norm{\delta_{(n-1)}(s)}{\Leb{2}_{w,z}} \left(\norm{\sigma_{(n)}}{\Leb{2}_{w,z}}+\norm{\sigma_{(n-1)}}{\Leb{\infty}_{w,z}} \right)\\
&\apprle {{c_0^{3}}} \action{F(s)}^3 \norm{\delta_{(n-1)}(s)}{\Leb{\infty}_{w,z}}+ c_0 \norm{\delta_{(n-1)}(s)}{\Leb{2}_{w,z}}
\end{align*}
and so it follows from  \eqref{eq:Kn-z-w} that 
\begin{align*}
|\nabla_z \delta\mathcal{K}_{(n)}(s,q)|&\apprle  c_0 \action{F(s)}^6\left(\norm{\delta_{(n-1)}(s)}{\Leb{\infty}_{w,z}} +\norm{\delta_{(n-1)}(s)}{\Leb{2}_{w,z}}\right).
\end{align*}
To estimate $\nabla_w \delta \mathcal{K}_{(n)}$, it follows from \eqref{eq:Kn-z-w} and \eqref{eq:bound-E0} that  
\begin{equation}\label{eq:delta-K-n}
\begin{aligned}
&\relphantom{=}|\nabla_w \delta \mathcal{K}_{(n)}(s,q)|\\
&\apprle f(s) |E_n(s,q)-E_{n-1}(s,q)|+\action{F(s)}|E_n(s,q)-E_{n-1}(s,q)| |\nabla E_0(q)|\\
&\apprle f(s) |E_n(s,q)-E_{n-1}(s,q)| +c_0^2 \action{F(s)} |E_n(s,q)-E_{n-1}(s,q)|.
\end{aligned}
\end{equation}
The second part is controlled by 
\[ 
  c_0^2 \action{F(s)}^6\left(\norm{\delta_{(n-1)}(s)}{\Leb{\infty}_{w,z}} +\norm{\delta_{(n-1)}(s)}{\Leb{2}_{w,z}}\right).
\]

To estimate the bound for the first term, we note that 
 \begin{equation}\label{eq:difference-density}
  \partial_s (\rho_{(n)}-\rho_{(n-1)})+\Div_q (\delta \boldj_n)=0,\quad \rho_{(n)}-\rho_{(n-1)}=\int_{\mathbb{R}^2} \gamma_{(n)}^2-\gamma_{(n-1)}^2 \myd{p}.
 \end{equation}
Following the argument as in the proof of Lemma \ref{lemma:Lip-E}, it follows from \eqref{eq:difference-density} that  
\begin{align*}
&\relphantom{=}|E_n(s,q)-E_{n-1}(s,q)|\\
&\apprle (1+s)\action{F(s)}^2 \norm{\delta \boldj_n(s)}{\Leb{\infty}_{q}}\\
&\relphantom{=}+(1+s)\action{F(s)}^2\norm{\delta_{(n-1)}(s)}{\Leb{2}_{w,z}}\left(\norm{\sigma_{(n)}(s)}{\Leb{2}_{z,w}}+\norm{\sigma_{(n-1)}(s)}{\Leb{2}_{w,z}}\right).
\end{align*}
On the other hand, we have
\begin{align*}
 \norm{\delta \mathbf{j}_n(s)}{\Leb{\infty}_q} &\apprle  \norm{\delta_{(n-1)}(s)}{\Leb{\infty}_{w,z}}\left[\norm{\action{p}^4\gamma_{(n)}(s)}{\Leb{\infty}_{q,p}}+\norm{\action{p}^4\gamma_{(n-1)}(s)}{\Leb{\infty}_{q,p}} \right]\\
 &\apprle \action{F(s)}^4 \norm{\delta_{(n-1)}(s)}{\Leb{\infty}_{w,z}} \left(\norm{\action{z}^4 \sigma_{(n)}(s)}{\Leb{\infty}_{w,z}}+\norm{\action{z}^4 \sigma_{(n-1)}(s)}{\Leb{\infty}_{w,z}} \right).
\end{align*}
Hence we have 
\begin{align*}
|E_n(s,q)-E_{n-1}(s,q)|&\apprle c_0 (1+s)\action{F(s)}^6\left( \norm{\delta_{n-1}(s)}{\Leb{2}_{w,z}}+\norm{\delta_{n-1}(s)}{\Leb{\infty}_{w,z}}\right)
\end{align*}
and therefore by \eqref{eq:delta-K-n}, we get 
\[ |\nabla_w \delta \mathcal{K}_{(n)}(s,q)|\apprle c_0 \action{F(s)}^6\left( \norm{\delta_{n-1}(s)}{\Leb{2}_{w,z}}+\norm{\delta_{n-1}(s)}{\Leb{\infty}_{w,z}}\right),\]
which completes the proof of Theorem \ref{thm:wave-operator}. 
\end{proof}

Finally, we are ready to prove Theorems \ref{thm:B} and \ref{thm:c}.
\begin{proof}[Proof of Theorem \ref{thm:B}]
Note that \eqref{eq:condition-mu-infinity} leads to \eqref{eq:bound-sigma0} by identifying $\mu_{\infty}$ with $\sigma_{1}$. By Theorem \ref{thm:wave-operator} with reverse time and using the lens transform, for $\varepsilon_0$ small enough the problem admits the global solution to \eqref{eq:VP-rep-potential-reformulate} on $[0,\infty)$. Moreover, $\mu(t=0)$ satisfies the assumption in Theorem \ref{thm:A} except that we only have the bound $\norm{\nabla_{x,v}\mu_0}{L^\infty}\apprle \varepsilon_0$ (without moment in $v$). This yields only the local uniform convergence stated in \eqref{eq:asymptotic-local-uniform}, thereby completing the proof of Theorem~\ref{thm:B}.
\end{proof}

\begin{proof}[Proof of Theorem \ref{thm:c}]
Following the previous argument, we can apply Theorem \ref{thm:B} to the reversed time to construct a map from $-\infty$ to $0$. Then we apply Theorem \ref{thm:A} at $t=0$ to construct a scattering map. This completes the proof of Theorem \ref{thm:c}.
\end{proof}

\begin{remark}
Without a harmonic potential, \cite{FOPW23} obtained the existence of wave operators for large final data in three dimensions by applying the Lions and Perthame result \cite{LP91}. A similar result was proved in \cite{HK25} for the Vlasov--Riesz systems of order $\alpha$ when $1<\alpha<1+\delta$. It is unclear whether we can construct wave operators for large final data.
\end{remark}

\section*{Acknowledgement}
The authors thank anonymous referees for valuable comments and suggestions, which have improved the manuscript.

\section*{Competing Interests}
The authors declare no conflict of interest.

\section*{AI disclosure statement}
The authors used generative AI to find a simplified proof of Proposition \ref{prop:interpolation}, which was shown in \cite[Proposition B.1]{HK25}, and to search classical literature on ODE theory. Also, it was used to find typos in the revision.

\bibliographystyle{amsplain}
\bibliography{refs}

\end{document}